\documentclass[onefignum,onetabnum]{siamart220329}

\usepackage{amsmath,amssymb,enumerate,graphicx,pifont,color,tikz,changepage,yfonts,scalerel,enumitem,kantlipsum}

\usepackage{pgfplots}
\usepackage{caption,subcaption}
\pgfplotsset{compat=1.17}
\usepgfplotslibrary{statistics}

\usepackage{bookmark}

\newsiamremark{remark}{Remark}

\usetikzlibrary{positioning}
\usetikzlibrary{cd}
\usetikzlibrary{external}
\newcommand{\datafile}{}

\usepackage{amssymb,amsfonts,amsmath}

\usepackage[mathscr]{eucal}

\usepackage{stmaryrd}

\usepackage{amsbsy}

\usepackage{bm}

\usepackage{braket}

\usepackage{mathtools}

\newcommand{\Tra}{{\sf T}}

\newcommand{\V}[2][]{{\bm{#1\mathbf{\MakeLowercase{#2}}}}} 
 
\newcommand{\M}[2][]{{\bm{#1\mathbf{\MakeUppercase{#2}}}}} 
 
\newcommand{\T}[2][]{\boldsymbol{#1\mathscr{\MakeUppercase{#2}}}} 

\newcommand{\Mz}[3][]{\M[#1]{#2}_{(#3)}}

\newcommand{\ca}{\oast}
\newcommand{\cl}{\varolessthan}
\newcommand{\cg}{\varogreaterthan}

\newcommand{\mb}[1]{\mathbb{#1}}
\newcommand{\bmat}[1]{\begin{bmatrix} #1 \end{bmatrix}}

\newcommand{\update}[1]{\textcolor{blue}{#1}}
 
\Crefname{ALC@unique}{Line}{Lines}

\graphicspath{{figs/}} 

\crefname{section}{\S}{\S}
\crefname{subsection}{\S}{\S}

\newcommand{\HOSVD}{HOSVD}
\newcommand{\STHOSVD}{STHOSVD} 
\newcommand{\rHOSVD}{rHOSVD} 
\newcommand{\rSTHOSVD}{rSTHOSVD}
\newcommand{\rHOSVDkron}{rHKron} 
\newcommand{\rSTHOSVDkron}{rSTHKron} 
\newcommand{\rHOSVDkronreuse}{rHKron-re}

\newcommand{\isTTM}{in-sequence multi-TTM}
\newcommand{\aaoTTM}{all-at-once multi-TTM}

\crefname{algorithm}{Alg.}{Algs.}

\title{Parallel Randomized Tucker Decomposition Algorithms\thanks{This work is supported by the National Science Foundation under
Grant No. CCF-1942892.}}
\author{Rachel Minster\thanks{Wake Forest University, Winston-Salem, NC (\email{minsterr@wfu.edu}, \email{ballard@wfu.edu})} \and Zitong Li\thanks{University of California Irvine, Irvine, CA (\email{zitongl5@uci.edu})} \and Grey Ballard\footnotemark[2]}

\begin{document}

\maketitle

\begin{abstract}
\addcontentsline{toc}{section}{Abstract}
The Tucker tensor decomposition is a natural extension of the singular value decomposition (SVD) to multiway data.
We propose to accelerate Tucker tensor decomposition algorithms by using randomization and parallelization.
We present two algorithms that scale to large data and many processors, significantly reduce both computation and communication cost compared to previous deterministic and randomized approaches, and obtain nearly the same approximation errors.
The key idea in our algorithms is to perform randomized sketches with Kronecker-structured random matrices, which reduces computation compared to unstructured matrices and can be implemented using a fundamental tensor computational kernel.
We provide probabilistic error analysis of our algorithms and implement a new parallel algorithm for the structured randomized sketch. 
Our experimental results demonstrate that our combination of randomization and parallelization achieves accurate Tucker decompositions much faster than alternative approaches. We observe up to a 16$\times$ speedup over the fastest deterministic parallel implementation on 3D simulation data.
\end{abstract}

\section{Introduction}
\label{sec:intro}

Tucker decompositions are low-rank tensor approximations capable of approximating multidimensional data with large compression rates while maintaining high accuracy. Large scale multidimensional data arises from many applications such as simulations of partial differential equations, data mining, facial recognition, and imaging. Processing these data requires computationally efficient methods. 
Randomized algorithms have been used to efficiently compute Tucker decompositions in works such as \cite{randtucker_survey,batselier2018computing,che2019randomized,che2021efficient,minster2020randomized,sun2020low,wolf2019low,zhou2014decomposition}, but the growing size of data is outpacing even randomized algorithms. 
Scaling these methods to handle large data calls for efficient parallelization. Many high-performance implementations of deterministic algorithms have been developed for Tucker decompositions \cite{BKK20,CC+17,KU16,choi2018high,LFB21}. 
We develop both sequential and parallel randomized algorithms that efficiently compute Tucker decompositions of large-scale multidimensional data by reducing both computation and communication compared to previous work. 

As we review in \cref{sec:background}, there are two dominant computational kernels to computing a Tucker decomposition: computing matrix singular value decompositions (SVD) and computing tensor-times-matrix (TTM) products. 
For the deterministic algorithms \HOSVD{} \cite{de2000multilinear} (\cref{alg:hosvd}) and \STHOSVD{} \cite{vann2012new} (\cref{alg:sthosvd}), computing the SVD is the typical bottleneck, and various methods trade off accuracy for reduced computational complexity. 
Our goal is to reduce the complexity of the SVD computation via randomization and remove it as the dominant cost without sacrificing too much accuracy. 
Existing randomized Tucker approaches, discussed in \cref{sec:related}, apply low-rank matrix approximation algorithms in place of matrix SVDs.
These matrix algorithms include randomized range finder \cite{halko2011finding} (RRF, see \cref{alg:rrf}), which computes part of the low-rank approximation and involves a slight overestimate of the target rank, or randomized SVD \cite{halko2011finding} (RandSVD, see \cref{alg:randsvd}), which involves a second pass over the data to obtain the final approximation with the exact target rank.

We propose two randomized algorithms in \cref{sec:seq_algs}, one based on \HOSVD{} and one based on \STHOSVD{}, which have comparable accuracy and running time.
In our algorithms, we use the RRF approach with Kronecker-structured random matrices, which reduces the computational complexity of the sketch compared to previous randomized approaches. 
As a significant added practical benefit, the Kronecker structure reduces the amount of random number generation compared to unstructured random matrices such as Gaussian. 
Furthermore, we propose a deterministic truncation of the resulting core (with overestimated ranks) in order to achieve the exact target ranks, obtaining the same effect as RandSVD-based approaches at much lower cost.
We show that our \HOSVD{}-based algorithm can be as computationally efficient as our \STHOSVD{}-based algorithm by employing a dimension tree optimization to avoid recomputation across sketches using memoization.

To accompany our algorithms, we develop probabilistic error guarantees in \cref{sec:erroranalysis} for a randomized matrix algorithm using a Kronecker product of random matrices. We use the matrix results to obtain theoretical guarantees for our Tucker algorithms.
Our bounds differ from previous results by accounting for the Kronecker structure and rank truncation in our algorithms and by reducing the probability of failure and amplification factors.

In \cref{sec:parallel}, we describe the parallelization of our proposed algorithms for distributed memory using the TuckerMPI library \cite{BKK20}, allowing us to scale the algorithms to large datasets that cannot be processed on a single server.
While previous work has combined randomization and parallelization, our implementation is the first to parallelize the randomized sketch, which significantly reduces the computational cost.
Moreover, in exploiting the Kronecker structure of our sketch, we implement a new parallel algorithm that communicates less data than the algorithm used by TuckerMPI and in fact minimizes interprocessor communication for the computation \cite{ABGKR22-TR-MTTM}.

Our experimental results are presented in \cref{sec:experiments}.
We validate the error guarantees of \cref{sec:erroranalysis} and show empirically that our structured random matrices are just as accurate as standard Gaussian random matrices.
Using synthetic data as well as two large simulation datasets, we demonstrate that our parallel randomized algorithms given in \cref{sec:parallel} scale well to thousands of cores and outperform alternative deterministic and randomized algorithms, achieving speedups of up to $16\times$ over the state-of-the-art implementation of the best deterministic algorithm.

\section{Background}\label{sec:background}
We first review the relevant background on tensors and randomized algorithms for matrices. For more details on tensors, see \cite{kolda2009tensor}, and for more details on randomized algorithms, see \cite{halko2011finding}.
\subsection{Tensor notation and operations}
A tensor $\T{X} \in \mb{R}^{n_1 \times n_2 \times \dots \times n_d}$ is a $d$-way array. We can unfold a $d$-mode tensor along each of its modes, or dimensions; the mode-$j$ unfolding, denoted $\M{X}_{(j)}$, is a matrix with columns formed as the mode-$j$ fibers of the tensor. Let $\| \cdot \|$ denote the tensor norm, which generalizes the matrix Frobenius norm. Since the following products will be frequently used to describe the sizes and ranks of a tensor, we define the following notations:
$n^{\ca} = \prod_{k=1}^{d}n_k$, $\, n^{\cl}_{i} = \prod_{k=1}^{i-1}n_k$, $\, n^{\cg}_{i} = \prod_{k=i+1}^{d}n_k$, $n^{\oslash}_{i} = \prod_{k\neq i}n_k$.

One key operation for tensors is the tensor-times-matrix product, or TTM. A tensor $\T{X} \in \mb{R}^{n_1 \times n_2 \times \dots \times n_d}$ is multiplied along mode $j$ by a matrix $\M{A} \in \mb{R}^{m \times n_j}$, denoted by $\T{X} \times_j \M{A}$, to obtain a tensor $\T{Y} \in \mb{R}^{n_1 \times \dots \times n_{j-1} \times m \times n_{j+1} \times \dots \times n_d}$. This product can also be expressed in terms of its mode-$j$ unfolding as $\Mz{Y}{j} = \M{A}\Mz{X}{j}$. We can also multiply a tensor $\T{X} \in \mb{R}^{n_1 \times \dots \times n_d}$ by up to $d$ matrices $\M{A}_j \in \mb{R}^{m_j \times n_j}$, $j = 1,\dots,d$, across distinct modes to obtain $\T{Y} = \T{X} \times_1 \M{A}_1 \times_2 \M{A}_2 \times \dots \times_d \M{A}_d \in \mb{R}^{m_1 \times \dots \times m_d}$. We call this product a multi-TTM; it is also known as a multilinear multiplication. If unfolded along mode $j$, we have 
$\Mz{Y}{j} = \M{A}_j \Mz{X}{j} \left( \M{A}_d \otimes \dots \otimes \M{A}_{j+1} \otimes \M{A}_{j-1} \otimes \dots \otimes \M{A}_1 \right)^\top,$
where $\otimes$ is the matrix Kronecker product.

\subsection{Tucker Decomposition}
\label{ssec:tucker}
The Tucker decomposition of a given tensor $\T{X} \in \mb{R}^{n_1 \times \dots \times n_d}$ of multirank $\V{r} = (r_1,\dots,r_d)$, where $r_j = \text{rank}(\Mz{X}{j})$ for each $j$, represents $\T{X}$ as the product of a core tensor $\T{G} \in \mb{R}^{r_1 \times \dots \times r_d}$ and $d$ factor matrices $\M{U}_j \in \mb{R}^{n_j \times r_j}$ such that $\T{X} = \T{G} \times_1 \M{U}_1 \times \dots \times_d \M{U}_d$. We can also obtain a low-rank approximation to $\T{X}$ in the Tucker form by taking the target rank $(r_1, \dots, r_d)$, or size of the core tensor $\T{G}$, to be less than the ranks of the unfoldings in each mode. 

\paragraph{Higher-Order SVD (HOSVD) and Sequentially Truncated HOSVD}
Two algorithms that compute low-rank Tucker decompositions of tensors are the higher-order SVD (HOSVD) \cite{de2000multilinear} and sequentially truncated HOSVD (STHOSVD) \cite{vann2012new}. The HOSVD algorithm forms each factor matrix $\M{U}_j$ from the first $r_j$ left singular vectors of the mode unfolding $\Mz{X}{j}$, and once all the factor matrices are computed, computes the core tensor as $\T{G} = \T{X} \times_1 \M{U}_1^\top \times \dots \times_d \M{U}_d^\top$ (see \cref{alg:hosvd}).

\vspace{-.4cm}
\hspace*{-\parindent}%
\begin{minipage}[t]{.47\textwidth}
\begin{algorithm}[H]
\caption{\HOSVD{} \cite{de2000multilinear}}
\label{alg:hosvd}
\begin{algorithmic}[1]\footnotesize
\Function{[$\T{G}, \{\M{U}_j\}] = $ \HOSVD}{$\T{X}, \V{r}$}
\For{$j=1:d$}
\State $\M{U}_j =$ first $r_j$ left sing. vecs. of $\Mz{X}{j}$
\EndFor
\State $\T{G} = \T{X} \times_1 \M{U}_1^\top \times \dots \times_d \M{U}_d^\top$
\EndFunction
\end{algorithmic}
\end{algorithm}
\end{minipage}
\begin{minipage}[t]{.51\textwidth}
\begin{algorithm}[H]
\caption{\STHOSVD{} \cite{vann2012new}}
\label{alg:sthosvd}
\begin{algorithmic}[1]\footnotesize
\Function{[$\T{G}, \{\M{U}_j\}] = $ \STHOSVD}{$\T{X}, \V{r}$}
\State $\T{G} = \T{X}$
\For{$j=1:d$}
\State $\M{U}_j =$ first $r_j$ left singular vecs. of $\Mz{G}{j}$
\State $\T{G} = \T{G} \times_j \M{U}_j^\top$
\EndFor
\EndFunction
\end{algorithmic}
\end{algorithm}
\end{minipage}
\vspace*{.2cm}

The STHOSVD algorithm is similar to HOSVD, but instead of handling all modes independently, it processes the modes in a predetermined sequence. 
After the first factor matrix is computed from the first $r_j$ left singular vectors of $\Mz{X}{j}$, we truncate in that mode by computing a partially truncated core tensor via a TTM with the factor matrix, $\T{G} \times_j \M{U}_j^\top$. We then use the partially truncated core $\T{G}$ for the next mode instead of the full tensor $\T{X}$, as shown in \cref{alg:sthosvd}.

\subsection{Randomized Matrix Algorithms}
The randomized range finder algorithm, made popular by \cite{halko2011finding} and shown in \cref{alg:rrf}, efficiently computes a low-rank representation of a matrix $\M{M} \in \mb{R}^{m \times n}$. 
Given a target rank $r$ and oversampling parameter $p$, we multiply $\M{M}$ by a random matrix $\M{\Omega} \in \mb{R}^{n \times \ell}$ with $\ell = r+p$ such that $\ell < m$, to form $\M{Y} \in \mb{R}^{m \times \ell}$, a matrix made up of random linear combinations of the columns of $\M{M}$. 
We then compute a thin QR decomposition of $\M{Y}$ to obtain a matrix $\M{Q} \in \mb{R}^{m \times \ell}$ whose range is a good estimate of the range of $\M{M}$. 
Projecting $\M{M}$ onto the range of $\M{Q}$ gives us the low-rank approximation $\M{M} \approx \M{QQ}^\top \M{M}$. 
We can choose any random distribution for random matrix $\M{\Omega}$. 
In this paper, we will consider both subsampled random Hadamard transform (SRHT) and standard Gaussian matrices.

Note that the resulting approximation from \cref{alg:rrf} is actually rank-$\ell$. 
If we seek a rank-$r$ approximation, further truncation is necessary. 
One way of truncating is to take a thin SVD of $\M{Q}^\top \M{M}$, which is the process taken in the randomized SVD algorithm in \cite{halko2011finding}, reproduced in \cref{alg:randsvd}.
We will adapt this truncation method in the algorithms developed in later sections.

\vspace{-.4cm}
\hspace*{-\parindent}%
\begin{minipage}[t]{.47\textwidth}
\begin{algorithm}[H]
\caption{Randomized Range Finder \cite{halko2011finding}}
\label{alg:rrf}
\begin{algorithmic}[1]\footnotesize
\Function{$\M{Q} = $ RandRangeFinder}{$\M{M},\M{\Omega}$}
\State $\M{Y} = \M{M \Omega}$
\State Compute thin QR $\M{Y} = \M{QR}$
\EndFunction
\end{algorithmic}
\end{algorithm}
\end{minipage}
\begin{minipage}[t]{.51\textwidth}
\begin{algorithm}[H]
\caption{Randomized SVD \cite{halko2011finding}}
\label{alg:randsvd}
\begin{algorithmic}[1]\footnotesize
\Function{$[\M{U},\M{\Sigma},\M{V}]=$ RandSVD}{$\M{M},r,\M{\Omega}$}
\State $\M{Q} = $ RandRangeFinder$(\M{M},\M{\Omega})$
\State $\M{B} = \M{Q}^\top \M{M}$
\State Compute thin SVD $\M{B} = \M[\hat]{U}\M{\Sigma}\M{V}^\top$
\State $\M{U} = \M{Q}\M[\hat]{U}(:,1:r)$
\State Truncate $\M{\Sigma} =(1:r,1:r)$, $\M{V} = \M{V}(:,1:r)$
\EndFunction
\end{algorithmic}
\end{algorithm}
\end{minipage}
\vspace{-.12cm}

\section{Related Work}
\label{sec:related}

Our work builds on three different categories of previous work, namely randomized algorithms for Tucker decompositions, probabilistic analysis of randomized algorithms, and parallel algorithms for tensor computations. 
\paragraph{Randomized Algorithms}
There has been much previous work on randomized algorithms for Tucker decompositions; a good survey of this work can be found in \cite{randtucker_survey}. The basic algorithms for randomized HOSVD and randomized STHOSVD are proposed in \cite{zhou2014decomposition}, while later work improves on the algorithms in various ways. One important distinction among randomized algorithms is the rank of the output approximation. In \cite{wolf2019low,zhou2014decomposition}, the approximation has rank $\V{\ell} = \V{r}+p$ as the randomized range finder (\cref{alg:rrf}) is used without additional truncation. Other algorithms, such as those presented in \cite{che2019randomized,sun2020low} do not oversample at all, limiting the potential accuracy of their methods. In \cite{che2021efficient}, the authors use the randomized range finder, but truncate by only taking the first $r_j$ columns of each factor matrix. The randomized SVD algorithm (\cref{alg:randsvd}) can be applied instead to both oversample and more accurately obtain the desired rank-$\V{r}$ approximation, which is done in \cite{batselier2018computing,minster2020randomized}. Our approach is most similar to the randomized SVD approach, but we apply it in a holistic manner, as discussed in \cref{ssec:rsimp}.

Another common improvement to the basic randomized algorithms comes from exploiting structure in the random matrices used to reduce storage and/or computational costs, as well as the number of random entries generated. 
Khatri-Rao products of random matrices are used in \cite{che2019randomized,sun2020low}, compact random matrices are employed in \cite{batselier2018computing}, and Kronecker products of random matrices are used in \cite{che2020computation,che2021efficient}. 
Our work is most similar to \cite{che2021efficient} as we also employ Kronecker products, but our algorithms improve upon those in \cite{che2021efficient} by truncating to the desired target rank in a more accurate manner. 
Kronecker product structure has also been exploited in other tensor decompositions besides Tucker decompositions: in \cite{battaglino2018practical,jin2021faster}, Kronecker products of random matrices are used to accelerate algorithms for CP decompositions; while in \cite{daas2021randomized}, Kronecker product structure was exploited in the context of the tensor-train decomposition. 
We also discuss how to implement our algorithms on distributed systems and provide improved probabilistic analysis. 

\paragraph{Error Analysis}
To accompany the discussed randomized algorithms, other work has developed probabilistic error analysis. 
Analysis for the standard version of randomized HOSVD is presented in \cite{erichson2020randomized,minster2020randomized}, and for the standard randomized STHOSVD algorithm in \cite{che2019randomized,minster2020randomized} for Khatri-Rao products of Gaussian matrices and dense Gaussian random matrices, respectively. 
Previous error analysis has been done for a randomized STHOSVD algorithm employing Kronecker products of subsampled randomized Fourier transform (SRFT) matrices  in \cite{che2021efficient}, but we make several improvements on this work. 
Our error bound, for our algorithms with Kronecker products of the real-valued equivalent of SRFTs, i.e., subsampled randomized Hadamard transform (SRHT) matrices, has an improved error constant and a smaller probability of failure. %

\paragraph{Parallel Algorithms} 
Our parallel algorithms and implementation, described in \cref{sec:parallel}, are built upon the foundation of TuckerMPI \cite{BKK20} and its improvements \cite{LFB21}.
TuckerMPI is a C++/MPI library that implements the STHOSVD algorithm to compute Tucker decompositions of large dense tensors that are distributed across machines.
It implements many other utilities such as file I/O and subroutines such as parallel TTM, that we use in our algorithms and experiments.
Other parallel implementations of Tucker algorithms have been developed for both dense \cite{CC+17} and sparse \cite{KU16} tensors.
The most similar work to ours combines parallelism and randomization to compute Tucker decompositions of dense tensors \cite{choi2018high}.
The approach taken by Choi, Liu, and Chakaravarthy \cite{choi2018high} is to employ STHOSVD (\cref{alg:sthosvd}) and compute the SVD of $\Mz{G}{j}$ by computing its Gram matrix in parallel and then sequentially applying randomized SVD (\cref{alg:randsvd}) to the Gram matrix.
Our approach differs in that we parallelize the randomized algorithm and avoid the Gram matrix computation; we provide a more detailed comparison in \cref{ssec:comparison}.

We propose a novel implementation of a parallel algorithm for the Multi-TTM computation in \cref{sec:aao-mttm}.
Communication lower bounds for this computation and theoretical algorithms that achieve those lower bounds are presented in \cite{ABGKR22-TR-MTTM}.
The Multi-TTM algorithm that we present in this paper can be seen as a specialization of \cite[Alg.~8.1]{ABGKR22-TR-MTTM}, but our implementation is novel.
We also highlight previous optimizations of tensor computations using dimension trees, a memoization technique.
First introduced in \cite{PTC13a} in the context of computing gradients of the CP decomposition optimization problem, dimension trees have also been used for Tucker decompositions.
For example, the Higher-Order Orthogonal Iteration (e.g., \cite{LDV00}) benefits from storing and reusing intermediate quantities across tensor modes as demonstrated in \cite{KR19}.
We use the dimension tree approach in a different context in one of our randomized algorithms; this process is described in \cref{sec:dimension_tree}.

\section{Sequential Algorithms}\label{sec:seq_algs}
We present our novel sequential algorithms before discussing how they may be implemented in parallel. 
There are several different variations of algorithms we will present, and we provide a hierarchy diagram for how they relate in \cref{fig:algs}. 
The first optimizations we show, using sequential truncation and randomization, have been developed in previous work. 
We then progress to using a Kronecker product of random matrices within the randomized algorithms, and then finally reusing Kronecker factors in the HOSVD case. 
The two most efficient algorithms we present are in the leftmost boxes of the two main subtrees: randomized STHOSVD with Kronecker products (\cref{alg:rsthosvdkron}) and randomized HOSVD with reused Kronecker factors (\cref{alg:rhosvdrekron}).
\setlength{\abovecaptionskip}{1pt}
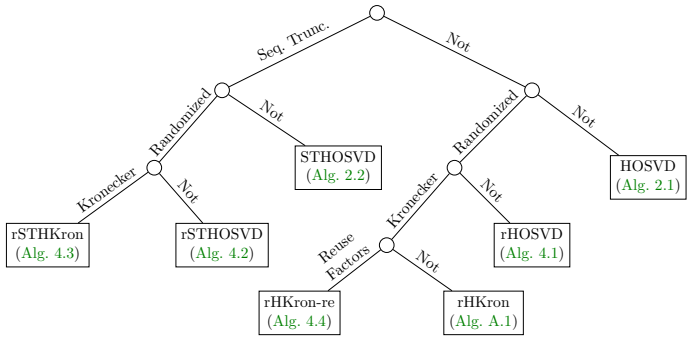
\begin{figure}[!ht]
\centering
\tikzexternaldisable
\tikzsetnextfilename{algs_hierarchy}

\tikzstyle{edgelabel} = [midway, above, sloped]
\resizebox{.7\textwidth}{!}{
\begin{tikzpicture}[scale=.875,namenode/.style={scale=.75}]


\node[draw,circle] (root) at (0,0) {};
\node[draw,circle] (notST) at (4,-2) {};
\node[draw,circle] (ST) at (-4,-2) {};
\node[draw,rectangle,align=center] (dHOSVD) at (7,-4.25) {\HOSVD{}\\(\cref{alg:hosvd})};
\node[draw,circle] (rHOSVD) at (2,-4) {};
\node[draw,rectangle,align=center] (dSTHOSVD) at (-1,-4) {\STHOSVD{}\\(\cref{alg:sthosvd})};
\node[draw,circle] (rSTHOSVD) at (-5.75,-4) {};
\node[draw,rectangle,align=center] (rSTHOSVDdense) at (-4,-6) {\rSTHOSVD{}\\(\cref{alg:rsthosvd})};
\node[draw,rectangle,align=center] (rSTHOSVDkron) at (-8.5,-6) {\rSTHOSVDkron{}\\(\cref{alg:rsthosvdkron})};
\node[draw,rectangle,align=center] (rHOSVDdense) at (4,-6) {\rHOSVD{}\\(\cref{alg:rhosvd})};
\node[draw,circle] (rHOSVDkron) at (.25,-6) {};
\node[draw,rectangle,align=center] (rHOSVDkronno) at (2.75,-7.75) {\rHOSVDkron{}\\(\cref{alg:rhosvdkron})};
\node[draw,rectangle,align=center] (rHOSVDkronreuse) at (-2,-7.75) {\rHOSVDkronreuse{}\\(\cref{alg:rhosvdrekron})};

\draw (root) -- (ST) node[edgelabel] {Seq.~Trunc.};
\draw (root) -- (notST) node[edgelabel] {Not};
\draw (notST) -- (dHOSVD) node[edgelabel] {Not};
\draw (notST) -- (rHOSVD) node[edgelabel] {Randomized};
\draw (ST) -- (dSTHOSVD) node[edgelabel] {Not};
\draw (ST) -- (rSTHOSVD) node[edgelabel] {Randomized};
\draw (rSTHOSVD) -- (rSTHOSVDdense) node[edgelabel] {Not};
\draw (rSTHOSVD) -- (rSTHOSVDkron) node[edgelabel] {Kronecker};
\draw (rHOSVD) -- (rHOSVDdense) node[edgelabel] {Not};
\draw (rHOSVD) -- (rHOSVDkron) node[edgelabel] {Kronecker};
\draw (rHOSVDkron) -- (rHOSVDkronno) node[edgelabel] {Not};
\draw (rHOSVDkron) -- (rHOSVDkronreuse) node[edgelabel, text width=40, text centered] {Reuse Factors};

\end{tikzpicture}}
\caption{Hierarchy of algorithms}
\label{fig:algs}
\end{figure}
\setlength{\abovecaptionskip}{10pt}

In all the algorithms presented in this section, we employ a holistic truncation approach. Given a tensor $\T{X}$, target rank $\V{r} = (r_1,\dots, r_d)$, and oversampling parameter $p$ and letting $\ell_j = r_j+p$ for $j = 1,\dots,d$, we first apply the randomized range finder algorithm (\cref{alg:rrf}) to each mode unfolding, and obtain an initial core $\hat{\T{G}} \in \mb{R}^{\ell_1 \times \dots \times \ell_d}$. We then apply the truncation phase of \cref{alg:randsvd} by computing a deterministic \STHOSVD{} of $\hat{\T{G}}$ such that $\T[\hat]{G} \approx \T{G} \times_1 \M{V}_1 \times \dots \times_d \M{V}_d$. 
The rank-$\V{r}$ representation of $\T{X}$ is then $\T{X} \approx \T{G} \times_1 \M[\hat]{U}_1\M{V}_1 \times \dots \times_d \M[\hat]{U}_d\M{V}_d$.

\subsection{Randomized HOSVD/STHOSVD}\label{ssec:rsimp}
The first algorithms we present are the basic form of randomized algorithms on which we improve throughout the paper. \cref{alg:rhosvd,alg:rsthosvd} are similar to other randomized projection algorithms for Tucker decompositions found in \cite{randtucker_survey,minster2020randomized} except for the truncation approach. Instead of directly applying the randomized SVD algorithm (\cref{alg:randsvd}) to each mode unfolding, we take the holistic approach described above. 
We can use this technique both with HOSVD, shown in \cref{alg:rhosvd}, and STHOSVD, shown in \cref{alg:rsthosvd}.

\vspace{-.4cm}
\hspace*{-\parindent}%
\begin{minipage}[t]{.49\textwidth}
\begin{algorithm}[H]
\caption{Randomized HOSVD}
\label{alg:rhosvd}
\begin{algorithmic}[1]\footnotesize
\Function{$[\T{G}, \{\M{U}_j\}] =$ \rHOSVD}{$\T{X},\V{r},p$}

\phantom{$\hat{\T{G}} = \T{X}$} 
\For{$j=1:d$}
\State Draw $\M{\Omega} \in \mb{R}^{n_j^\oslash \times \ell_j}$
\State $\hat{\M{U}}_j = $ RandRangeFinder$(\Mz{X}{j},\M{\Omega})$
\EndFor
\State $\hat{\T{G}} = \T{X} \times_1 \hat{\M{U}}_1^\top \times \dots \times_d \hat{\M{U}}_d^\top$
\State $[\T{G},\{\M{V}_j\}] = \text{\STHOSVD}(\hat{\T{G}},\V{r})$
\State $\M{U}_j = \hat{\M{U}}_j \M{V}_j$ for $j = 1,\dots,d$
\EndFunction
\end{algorithmic}
\end{algorithm}
\end{minipage}
\begin{minipage}[t]{.49\textwidth}
\begin{algorithm}[H]
\caption{Randomized STHOSVD}
\label{alg:rsthosvd}
\begin{algorithmic}[1]\footnotesize
\Function{$[\T{G}, \{\M{U}_j\}] =$ \rSTHOSVD}{$\T{X},\V{r},p$}
\State $\hat{\T{G}} = \T{X}$
\For{$j=1:d$}
\State Draw $\M{\Omega} \in \mb{R}^{r_j^\cl n_j^\cg \times \ell_j}$
\State $\hat{\M{U}}_j = $ RandRangeFinder$(\Mz[\hat]{G}{j},\M{\Omega})$
\State $\hat{\T{G}} = \hat{\T{G}} \times_j \hat{\M{U}}_j^\top$
\EndFor
\State $[\T{G},\{\M{V}_j\}] = \text{\STHOSVD}(\hat{\T{G}},\V{r})$
\State $\M{U}_j = \hat{\M{U}}_j \M{V}_j$ for $j = 1,\dots,d$
\EndFunction
\end{algorithmic}
\end{algorithm}
\end{minipage}

\subsection{Randomized HOSVD/STHOSVD with Kronecker product}
Our main algorithm combines the HOSVD/STHOSVD algorithms with a special case of the randomized range finder algorithm used on each mode unfolding. 
Within the randomized range finder, we will represent the random matrix $\M{\Omega}$ as a Kronecker product of random matrices each with a small number of columns instead of a single large $\M{\Omega}$. 
This allows us both to employ a Multi-TTM operation instead of matrix multiplication, reducing the computational complexity, and to exploit properties of tall-and-skinny matrices in our parallel algorithms. Specifically, for a tensor $\T{X} \in \mb{R}^{n_1 \times \dots \times n_d}$, rank $\V{r} = (r_1,\dots,r_d)$, and oversampling parameter $p$ with $\ell_j = r_j+p$,
define $\M{\Omega}_j \in \mb{R}^{n_j^\oslash \times \ell_j}$ as 
	$\M{\Omega}_j = \left(\M{\Phi}_{j,d} \otimes  \dots \otimes \M{\Phi}_{j,j+1} \otimes \M{\Phi}_{j,j-1} \otimes \dots \otimes \M{\Phi}_{j,1}\right)^\Tra,$
with $\M{\Phi}_{j,k} \in \mb{R}^{s_{j,k} \times n_k}$ a random matrix from some distribution (e.g. Gaussian, SRHT, etc.), and $\M{S} \in \mb{N}^{d \times d}$ a matrix of subranks. We define $\M{S}$ to have entries $s_{j,k}$ the $k$-th subrank for mode $j$ where $k \neq j$, and diagonal entries $s_{j,j} = 1$ for $j = 1,\dots,d$, such that the row products $\prod_{k =1}^{d} s_{j,k} = \ell_j$, for $j = 1,\dots, d$. 
We summarize the steps for our algorithm in STHOSVD form in \cref{alg:rsthosvdkron} (\rSTHOSVDkron), and include an HOSVD version in \cref{alg:rhosvdkron}.
Note that \cref{line:omega_stkron,line:sketch_stkron,line:qr_stkron} in \cref{alg:rsthosvdkron} consist of applying randomized range finder to the current mode unfolding using the Kronecker product as our random matrix.

\begin{algorithm}[!ht]
\caption{Randomized STHOSVD with Kronecker product}
\label{alg:rsthosvdkron}
\begin{algorithmic}[1]\footnotesize
\Function{$[\T{G}, \{\M{U}_j\}] =$ \rSTHOSVDkron}{$\T{X},\V{r},p$}
\State $\T[\hat]{G}=\T{X}$
\State Compute matrix of subranks $\M{S}$
\For{$j = 1:d$}
\State\label{line:omega_stkron} Draw $d-1$ random matrices $\M{\Phi}_{j,k} \in \mb{R}^{s_{j,k} \times \ell_k}$ for $k< j$ and $\M{\Phi}_{j,k} \in \mb{R}^{s_{j,k} \times n_k}$ for $k >j$
\State\label{line:sketch_stkron} $\T{Y} \leftarrow \T[\hat]{G} \times_1 \M{\Phi}_{j,1} \times \dots \times_{j-1} \M{\Phi}_{j,j-1} \times_{j+1} \M{\Phi}_{j,j+1} \times \dots \times_d \M{\Phi}_{j,d}$
\State\label{line:qr_stkron} Compute thin QR $\Mz{Y}{j} = \M[\hat]{U}_j \M{R}$
\State\label{line:core_stkron} $\T[\hat]{G} = \T[\hat]{G} \times_j \M[\hat]{U}_j^\Tra$
\EndFor
\State\label{line:truncatecore_stkron} $[\T{G}, \{\M{V}_j\}] = $ STHOSVD$(\hat{\T{G}}, \V{r})$
\State $\M{U}_j = \hat{\M{U}}_j \M{V}_j$ for $j = 1,\dots, d$
\EndFunction
\end{algorithmic}
\end{algorithm}

\paragraph{Choosing subranks} The restriction on the subranks $\M{s}$ is that $\prod_{k \neq j}^{d} s_{j,k} = \ell_j = r_j+p$ for each row $j = 1,\dots d$. 
In practice, we can actually choose the subranks so that $\prod_{k \neq j}^{d} s_{j,k} \geq \ell_j$. 
Satisfying this looser condition means we are essentially increasing the oversampling we are already doing through parameter $p$. 
This frees us to use heuristics to choose the subranks. 
We choose each row of $\M{S}$ to be composed of $d-1$ integer factors of $\ell_j$. In the case that $\ell_j$ does not have exactly $d-1$ integer factors, we adjust the oversampling parameter until we can obtain the correct number of factors.

\subsection{Randomized HOSVD with Kronecker Factor Reuse}
\label{sec:reuse}
An additional adaption we make to reduce the number of random values generated and computation is to reuse the components of the Kronecker product $\M{\Omega}$. 
Instead of generating $d-1$ random matrices for each mode as in 
\cref{alg:rsthosvdkron}, we generate $d$ random matrices $\{\M{\Phi}_j\}$ once, and use different combinations of $d-1$ of those same matrices in each mode. 
This approach is summarized in \cref{alg:rhosvdrekron}. 
One benefit of this approach is that we generate significantly fewer random entries. 
We can also, as will be addressed in \cref{sec:parallel}, implement dimension trees to reduce computational cost. 
This variation only works for the HOSVD approach as the size of each $\M{\Phi}_j$ remains the same, while it would change after each mode in an STHOSVD approach.
\setlength\intextsep{0.01cm}
\begin{algorithm}[!ht]
\caption{Randomized HOSVD with Kronecker product re-using factors}
\label{alg:rhosvdrekron}
\begin{algorithmic}[1]\footnotesize
\Function{$[\T{G}, \{\M{U}_j\}] =$ \rHOSVDkron}{$\T{X},\V{r},p$}
\State Compute subranks $\V{s}$
\State Draw $d$ random matrices $\M{\Phi}_k \in \mb{R}^{s_k \times n_k}$ for $k = 1,\dots, d$
\For{$j = 1:d$}
\State $\T{Y} \leftarrow \T{X} \times_1 \M{\Phi}_1 \times \dots \times_{j-1} \M{\Phi}_{j-1} \times_{j+1} \M{\Phi}_{j+1} \times \dots \times_d \M{\Phi}_d$\label{line:sketchMultiTTM}
\State Compute thin QR $\Mz{Y}{j} = \M{U}_j \M{R}$
\EndFor
\State $\hat{\T{G}} = \T{X} \times_1 \hat{\M{U}}_1^\top \times \dots \times_d \hat{\M{U}}_d^\top$
\State $[\T{G}, \{\M{V}_j\}] = $ \STHOSVD$(\hat{\T{G}}, \V{r})$
\State $\M{U}_j = \hat{\M{U}}_j \M{V}_j$ for $j = 1,\dots, d$
\EndFunction
\end{algorithmic}
\end{algorithm}
\setlength\intextsep{0.3cm}
\paragraph{Choosing subranks} 
Note that in this case we compute only one vector of subranks $\V{s} \in \mb{N}^d$, instead of a matrix as in \cref{alg:rsthosvdkron}. 
We compute these subranks heuristically as well, and in this case in a straightforward manner, deriving the formula $ s_i = \lceil ( \prod_{j=1}^d \ell_j )^{\frac{1}{d-1}} / \ell_i \rceil$,
from the conditions $\V{s}\in \mb{N}^d$ and $\prod_{k=1}^{d-1} s_{k} \geq \ell_j$ for $j = 1,\dots,d$. 
This formula, while more straightforward, is more constricting than the method we use to compute subranks for \cref{alg:rsthosvdkron} because $s_j\ell_j$ is approximately fixed for each mode $j$.
This can create problems when computing with tensors with skewed modes and ranks, so we recommend not reusing Kronecker factors in these cases.

\subsection{Dimension Tree Optimization}\label{sec:dimension_tree}
In \cref{line:sketchMultiTTM} of \cref{alg:rhosvdrekron}, we perform a multi-TTM to compute the sketch tensor $\T{Y}$ for each mode, resulting in $d$ multi-TTM products. 
Notice, however, that the $d-1$ random matrices that we use in each multi-TTM are drawn from the same set of $d$ random matrices $\{\M{\Phi}_1,\dots, \M{\Phi}_d\}$. Thus a significant number of computations in each multi-TTM are repeated and can be reused. 
The dimension tree concept, which has been employed to reduce computational complexity in other algorithms \cite{CC+17,KR19}, also applies in this situation. 
In our implementation, we use a binary tree with $d$ leaf nodes for a $d$-way tensor. 
For example, given a 4-way tensor $\T{X}$ and four random matrices $\{\M{\Phi}_1,\M{\Phi}_2,\M{\Phi}_3,\M{\Phi}_4\}$, the 4 multi-TTM operations that would be carried out without a dimension tree are shown in the four leaf nodes of \cref{fig:dimTree}. 
The dimension tree provides an efficient way to perform the computations shared between each pair of adjacent leaf nodes and store the results in memory to reduce computation.  
For a $d$-way tensor $\T{X} \in \mb{R}^{n \times \dots \times n}$ and $d$ random matrices $\{\M{\Phi}_j\}$ each of size $r\times n$, the cost of using dimension trees to sketch every unfolding of $\T{X}$ is shown in \cref{tab:dimTreeCost}. When $r\ll n$, both costs are approximated by the first terms in the summations. 
As $d$ increases, the cost reduction that comes from using a dimension tree increases proportionally to $d/2$. 
\setlength{\belowcaptionskip}{-10pt}
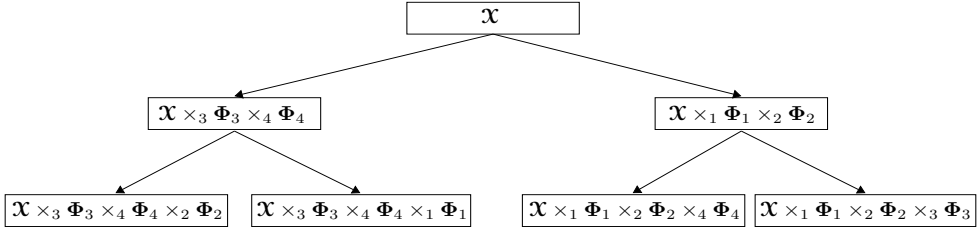
\begin{figure}[h]
    \centering
    \tikzexternaldisable
	\tikzsetnextfilename{dimTree}
    \resizebox{1.\textwidth}{.23\textwidth}{
\begin{tikzpicture}[x=0.75pt,y=0.75pt,yscale=-1,xscale=1,every node/.style={scale=0.95}]

    \draw   (248,24) -- (348,24) -- (348,44) -- (248,44) -- cycle ;
    \draw   (98,84) -- (198,84) -- (198,104) -- (98,104) -- cycle ;
    \draw   (392,84) -- (492,84) -- (492,104) -- (392,104) -- cycle ;
    \draw   (15,145) -- (144,145) -- (144,165) -- (15,165) -- cycle ;
    \draw   (158,145) -- (285,145) -- (285,165) -- (158,165) -- cycle ;
    \draw   (315,145) -- (444,145) -- (444,165) -- (315,165) -- cycle ;
    \draw   (450,145) -- (578,145) -- (578,165) -- (450,165) -- cycle ;
    \draw[-{Triangle}]    (298,44) -- (148,83) ;
    \draw[-{Triangle}]    (298,44) -- (442,83) ;
    \draw[-{Triangle}]    (148,105) -- (221,144) ;
    \draw[-{Triangle}]    (148,105) -- (79,144) ;
    \draw[-{Triangle}]    (442,105) -- (514,144) ;
    \draw[-{Triangle}]    (442,105) -- (379,144) ;

    \draw (290,27) node [anchor=north west][inner sep=0.75pt]   [align=left] {$\T{X}$};
    \draw (103,88) node [anchor=north west][inner sep=0.75pt]   [align=left] {$\T{X}\times_3\M{\Phi}_3\times_4\M{\Phi}_4$};
    \draw (399,88) node [anchor=north west][inner sep=0.75pt]   [align=left] {$\T{X}\times_1\M{\Phi}_1\times_2\M{\Phi}_2$};
    \draw (18,148) node [anchor=north west][inner sep=0.75pt]   [align=left] {$\T{X}\times_3\M{\Phi}_3\times_4\M{\Phi}_4\times_2\M{\Phi}_2$};
    \draw (160,148) node [anchor=north west][inner sep=0.75pt]   [align=left] {$\T{X}\times_3\M{\Phi}_3\times_4\M{\Phi}_4\times_1\M{\Phi}_1$};
    \draw (318,148) node [anchor=north west][inner sep=0.75pt]   [align=left] {$\T{X}\times_1\M{\Phi}_1\times_2\M{\Phi}_2\times_4\M{\Phi}_4$};
    \draw (452,148) node [anchor=north west][inner sep=0.75pt]   [align=left] {$\T{X}\times_1\M{\Phi}_1\times_2\M{\Phi}_2\times_3\M{\Phi}_3$};

    \end{tikzpicture}}
    \caption{Dimension tree for computing the sketches of a 4-way tensor $\T{X}$ in \cref{alg:rhosvdrekron}.} 
    \label{fig:dimTree}
\end{figure}
\begin{table}[!ht]
    \centering
\begin{tabular}{c|c|c}
 & \textbf{with dimTree} & \textbf{without dimTree} \\ \hline
 $\bm{d = 3}$ & $2(2rn^3 + 3r^2n^2)$ & $2(3rn^3 + 3r^2n^2)$ \\ \hline
 $\bm{d=4}$ & $2(2rn^4 + 2r^2n^3 + 4r^3n^2)$ & $2(4rn^4 + 4r^2n^3 + 4r^3n^2)$ \\ \hline
 $\bm{d=5}$ & $2(2rn^5 + 2r^2n^4 + 3r^3n^3 + 5r^4n^2)$ & $2(5rn^5 + 5r^2n^4 + 5r^3n^3 + 5r^4n^2)$
 \end{tabular}
    \caption{Complexity reduction examples that can be achieved using dimension trees.}
    \label{tab:dimTreeCost}
\end{table}

\subsection{Computational Complexity}

To analyze the computational cost of our algorithms, we consider the notationally simpler case with a $d$-mode tensor $\T{X} \in \mb{R}^{n \times \dots \times n}$, target rank $(r,\dots,r)$, and oversampling parameter $p$, letting $\ell=r+p$. We will also let the subranks $\{s_{j,k}\}$ and $\{s_k\}$ all be the same value, which we denote as $s=\ell^{d-1}$. Assume $s < r \ll n$.

There are two dominant costs for each algorithm presented: computing an SVD for each mode, and forming the core tensor via a multi-TTM or a series of TTMs (in the \STHOSVD{} case). We show the leading terms of both dominant steps for the standard algorithms, \cref{alg:hosvd,alg:sthosvd}, and compare to those for \cref{alg:rhosvd,alg:rsthosvd,alg:rsthosvdkron,alg:rhosvdrekron} in \cref{tab:compcost}. More details on the cost analysis for \cref{alg:hosvd,alg:sthosvd} can be found in \cite{vann2012new}, while more details on the analysis for \cref{alg:rhosvd} and \cref{alg:rsthosvd} can be found in \cite{minster2020randomized}. Based on the terms shown in \cref{tab:compcost}, we advocate the use of either \cref{alg:rsthosvdkron} or \cref{alg:rhosvdrekron}, and we show more detailed analysis on these two algorithms in \cref{sec:seq_complexity}.

\begin{table}[!ht]
\centering
\begin{tabular}{l|l|l}
 &  \multicolumn{2}{c}{\textbf{Leading term}} \\
 \hline
 \textbf{Algorithm} & \textbf{SVD} & \textbf{TTM} \\
 \hline 
HOSVD (\ref{alg:hosvd}) & $dn^{d+1}$ & $2rn^d$ \\
STHOSVD (\ref{alg:sthosvd}) & $n^{d+1}$ & $2rn^d$ \\ 
rHOSVD (\ref{alg:rhosvd}) & $2d\ell n^d$ & $2\ell n^d$ \\
rSTHOSVD (\ref{alg:rsthosvd}) & $2\ell n^d$ & $2\ell n^d$ \\
rSTHOSVDkron (\ref{alg:rsthosvdkron}) & $2\ell^{\frac{1}{d-1}} n^d$ & $2\ell n^d$ \\ 
rHOSVDkronreuse (\ref{alg:rhosvdrekron}) & $4\ell^{\frac{1}{d-1}} n^d$ & $2\ell n^d$ \\
\end{tabular}
\caption{Computational Cost}
\label{tab:compcost}
\end{table}

\subsection{Comparison with Previous Work}

We compare our sequential algorithms with previous approaches, in particular those based on the use of randomized SVD \cite{minster2020randomized} and Kronecker-structured random matrices \cite{che2021efficient}.
As summarized in \cite{randtucker_survey}, randomization can be used in multiple ways to compute Tucker approximations of tensors.
The most similar approaches to \cref{alg:rhosvd,alg:rsthosvd} are \cite[Algs. 3.1 and 3.2]{minster2020randomized}.
These algorithms replace the deterministic SVD within \cref{alg:hosvd,alg:sthosvd} with randomized SVD (\cref{alg:randsvd}).
The randomized SVD requires computing the thin SVD of the projection of the approximate column space; it increases the computational cost compared to randomized range finder by a constant factor greater than 2.
Because \cref{alg:rhosvd,alg:rsthosvd} use randomized range finder (\cref{alg:rrf}), they involve only one operation with the input (the random sketch) at the expense of working with the oversampled rank $\V{\ell}$ rather than the target rank $\V{r}$ until the final core truncation step.
The oversampled ranks are only slightly larger than the target ranks, so \cref{alg:rhosvd,alg:rsthosvd} are computationally cheaper than \cite[Algs. 3.1 and 3.2]{minster2020randomized}.

The Kronecker-structured random sketches of \cref{alg:rsthosvdkron,alg:rhosvdrekron} are similar to  \cite[Alg. 3.1]{che2021efficient}.
This algorithm uses a Sampled Random Fourier Transform to sketch each mode's matricization within STHOSVD.
Besides using a complex-valued random matrix, the key difference with \cref{alg:rsthosvdkron} is the truncation strategy: after computing the thin QR decomposition of the sketched matrix with $\ell$ columns, all but the first $r$ columns are truncated.
As we demonstrate in \cref{sec:accuracy}, our truncation strategy using a deterministic STHOSVD of the core tensor computes a more accurate approximation.

To the best of our knowledge, reusing Kronecker factors and exploiting the possible memoization of temporary quantities as presented in \cref{sec:reuse,sec:dimension_tree} have not been considered before.

\section{Error Analysis}\label{sec:erroranalysis}
We now present error guarantees for \cref{alg:rsthosvdkron,alg:rhosvdrekron}. Let $\T{T}=[\T{G}; \M{U}_1,\dots, \M{U}_d]$ be the approximation from either algorithm, and $\T[\hat]{T}=[\T[\hat]{G}; \M[\hat]{U}_1,\dots, \M[\hat]{U}_d]$ be the intermediate rank-$\V{\ell}$ approximation. The overall form of the error is
$$\varepsilon_{\text{total}} = \| \T{X} - \T{T} \|  \leq \| \T{X} - \T[\hat]{T}\| + \| \T[\hat]{T} - \T{T} \| = \varepsilon_{\text{rand}} + \varepsilon_{\text{core}},$$
where $\varepsilon_{\text{rand}}$ represents the error from forming the rank-$\V{\ell}$ approximation, and $\varepsilon_{\text{core}}$ represents the error in truncating the approximation to rank $\V{r}$.
The component $\varepsilon_{\text{core}}$ is equivalent to the error in computing the STHOSVD of $\T[\hat]{G}$, which we can see from
\begin{equation*}
	\varepsilon_{\text{core}} = \| \left( \T[\hat]{G} - \T{G} \times_1 \M{V}_1 \times \dots \times_d \M{V}_d \right) \times_1 \M[\hat]{U}_1 \times \dots \times_d \M[\hat]{U}_d \| = \| \T[\hat]{G}- \T{G} \times_1 \M{V}_1 \times \dots \times_d \M{V}_d \|,
\end{equation*}
where the second equality follows from the orthonormality of $\{\M[\hat]{U}_j\}$.
We can then apply the error bound for STHOSVD \cite[Theorem 6.5]{vann2012new} to initial core $\T[\hat]{G}$, i.e,
$\varepsilon_{\text{core}}^2 \leq \sum_{j=1}^d \sum_{i=r_j+1}^{\ell_j} \sigma_i^2\left(\Mz{\hat{G}}{j} \right)$,
where $\sigma_i(\M{A})$ denotes the $i$th singular value of $\M{A}$.

 Then, as $\T[\hat]{G}$ is a random quantity, we need to relate the singular values of $\Mz{\hat{G}}{j}$ to the singular values of $\Mz{X}{j}$ to get a deterministic upper bound. 
Because matrices $\{\M[\hat]{U}_j\}$ are orthonormal, the mode-wise singular values of $\T[\hat]{G}$ cannot be larger than those of $\T{X}$.
 A formal proof of this fact can be found in \cref{sec:appendix_core}.
 Thus, we have
\begin{equation}\label{eq:core_svals}
\varepsilon_{\text{core}}^2 \leq \sum_{j=1}^d \sum_{i=r_j+1}^{\ell_j} \sigma_i^2\left(\Mz{X}{j} \right).
\end{equation}

The rest of this section considers the component $\varepsilon_{\text{rand}}$.
Starting with an error bound for randomized range finder (\cref{alg:rrf}) using a Kronecker product of SRHT matrices, we extend the results to an error bound for our HOSVD-type algorithm (\cref{alg:rhosvdrekron}) and discuss how to adapt the proof for our STHOSVD-type algorithm (\cref{alg:rsthosvdkron}).

\subsection{Matrix Bound}
\paragraph{Random Matrix}
Let $n = \prod_{j=1}^q n_j$ and $\ell = \prod_{j=1}^q s_j$. We will consider a Kronecker product of SRHT matrices of the form 
\begin{equation}\label{eq:srht_kron}
 \M{\Omega} = \M{D}(\M{H}_1\otimes \M{H}_2 \otimes \dots \otimes \M{H}_q) \in \mb{R}^{n \times \ell},
 \end{equation}
where $\M{D} \in \mb{R}^{n \times n}$ is a diagonal matrix with i.i.d. entries from the Rademacher distribution, i.e. either $1$ or $-1$ with equal probability, and for each $j= 1,\dots, q$, $\M{H}_j \in \mb{R}^{n_j \times s_j}$ is formed from $s_j$ randomly sampled columns of an $n_j \times n_j$ Walsh-Hadamard matrix scaled by $\frac{1}{\sqrt{n_j}}$. Due to this scaling, $\M{\Omega}$ is orthonormal. Also note that the $\{\M{H}_j\}$ matrices are generated independently.

\paragraph{Notation}
We now introduce the setup and notation for our main theorem. Following the notation of \cite{halko2011finding}, let $\M{X} = \M{U \Sigma V}^\top$ be the SVD of matrix $\M{X} \in \mb{R}^{m \times n}$ with $m\leq n$. 
Fix target rank $r$, and partition the SVD as 
\begin{equation}\label{eq:partition}
\M{X} = \M{U}\bmat{ \M{\Sigma}_1 & \\ & \M{\Sigma}_2} \bmat{\M{V}_1^\top \\ \M{V}_2^\top},
\end{equation}with $\M{\Sigma}_1 \in \mb{R}^{r \times r}, \M{\Sigma}_2 \in \mb{R}^{(m-r) \times (m-r)}, \M{V}_1 \in \mb{R}^{n \times r}$, and $\M{V}_2 \in \mb{R}^{n \times (m-r)}$ such that $\M{V}_1$ and $\M{V}_2$ have orthonormal columns.
Now let $\M{\Omega} \in \mb{R}^{n \times \ell}$ be the random matrix defined in \cref{eq:srht_kron}, and define 
\vspace{-.2cm}
\begin{equation} \label{eq:omegas}
\M{\Omega}_1 = \M{V}_1^\top \M{\Omega}, \quad \M{\Omega}_2 = \M{V}_2^\top \M{\Omega}.
\vspace{-.1cm}
\end{equation}

We are interested in bounding the largest singular values of $\M{\Omega}_1^\dagger$ and $\M{\Omega}_2$. We will use the orthonormality of $\M{\Omega}_2$ to bound its largest singular values, but the argument is more complicated for $\M{\Omega}_1$.
Here we will adapt the approach in \cite{matrixproof}, allowing for the application to a Kronecker product of independent SRHT matrices and making other minor improvements. 
This bound is stated in \cref{lem:omega1}, which we will then use to prove our approximation error bound. 
We prove \cref{lem:omega1} in \cref{sec:appendix_omega1}.

\begin{lemma}\label{lem:omega1}
	Let $\M{\Omega}_1 \in \mb{R}^{r \times \ell}$ as defined in \cref{eq:omegas}, with $\M{\Omega} \in \mb{R}^{n \times \ell}$ the Kronecker product of $q$ SRHT matrices as defined in \cref{eq:srht_kron}, where $n = \prod_{j=1}^q n_j$ and $\ell = \prod_{j=1}^q s_j$. 
	Let $\alpha, \beta > 1$ be real numbers that satisfy
	\vspace{-.15cm}
\begin{equation}\label{eq:alphabeta}
\min_k\{s_k\} \geq \frac{\alpha^2\beta}{(\alpha-1)^2}(r^2+r).
\end{equation}
Then $\frac{1}{\sigma_{\min}^2(\M{\Omega}_1)} \leq \frac{\alpha n}{\ell}$,
with probability at least $1-\frac{1}{\beta^2}$.
\end{lemma}

\begin{remark}
	The bound in \cref{lem:omega1} contains a factor of $n$, the number of rows of Kronecker product $\M{\Omega}$. 
	This factor is not present in singular value bounds for Gaussian random matrices, and is a consequence of using SRHT matrices, as seen in bounds involving a single SRHT matrix in \cite{tropp2011improved}. 
	We choose to use SRHT instead of Gaussian matrices in our context because the Kronecker product of Gaussian matrices is no longer Gaussian, while the Kronecker product of SRHT matrices retains some SRHT properties. 
	In the empirical results shown in \cref{sec:accuracy}, we see that using SRHT matrices produces similar accuracy to Gaussian matrices; we anticipate that future work in random matrix theory will be able to improve this factor.
\end{remark}

\begin{theorem}\label{thm:matrixbound}
Let $\hat{\M{X}} = \M{QQ}^\top \M{X}$ be the approximation given by the randomized range finder of matrix $\M{X} \in \mb{R}^{m \times n}$ with target rank $r$, oversampling parameter $p$ such that $\ell=r+p \leq \min\{m,n\}$, and random sampling matrix $\M{\Omega}$ as defined in \cref{eq:srht_kron}. 
Let $\alpha,\beta > 1$ satisfy \cref{eq:alphabeta}.
Then, with probability at least $1-\frac{1}{\beta^2}$,
\vspace{-.1cm}
\begin{equation*}
	\| \M{X} - \M{QQ}^\top \M{X} \|_F \leq \left( \left(1+ \frac{\alpha n}{\ell} \right)\sum_{i=r+1}^{\min\{m,n\}} \sigma_i^2(\M{X}) \right)^{1/2}.
	\vspace{-.1cm}
\end{equation*}
\end{theorem}

\begin{proof}
Immediately from \cite[Theorem 9.1]{halko2011finding} and recalling the partitioning from \cref{eq:partition} and \cref{eq:omegas}, we have
\begin{equation}\label{eq:hmt}
\begin{aligned}
	\| \M{X} - \M{QQ}^\top \M{X} \|_F^2 = \| (\M{I} - \M{QQ}^\top ) \M{X} \|_F^2 &\leq \| \M{\Sigma}_2 \|_F^2 + \| \M{\Sigma}_2 \M{\Omega}_2 \M{\Omega}_1^\dagger \|_F^2 \\ &\leq \left( 1 + \|\M{\Omega}_2 \|_2^2 \|\M{\Omega}_1^\dagger\|_2^2 \right) \| \M{\Sigma}_2\|_F^2,
\end{aligned}
\end{equation}
We can apply the singular value bounds from \cref{lem:omega1} to obtain the bounds for $\| \M{\Omega}_1^\dagger\|_2^2$. With probability at least $1-\frac{1}{\beta^2}$, $\| \M{\Omega}_1^\dagger \|_2^2 = \frac{1}{\sigma^2_{\text{min}}(\M{\Omega}_1)} \leq \frac{\alpha n}{\ell}$.
For $\| \M{\Omega}_2\|_2^2$, we use properties of both $\M{V}_2$ and $\M{\Omega}$: $\| \M{\Omega}_2 \|_2 = \| \M{V}_2^\top \M{\Omega} \|_2 \leq \| \M{\Omega} \|_2 = 1$, as $\M{V}_2^\top $ has orthonormal rows, and $\M{\Omega}$ has orthonormal columns as the Kronecker product of matrices with orthonormal columns.
Combining these bounds, we obtain $\|\M{\Omega}_2 \|_2^2 \|\M{\Omega}_1^\dagger\|_2^2 \leq \frac{\alpha n}{\ell}$ with probability at least $1-\frac{1}{\beta^2}$.
From \cref{eq:hmt}, we now have
$\| \M{X} - \M{QQ}^\top \M{X} \|_F^2 \leq \left( 1+ \frac{\alpha n}{\ell} \right) \| \M{\Sigma}_2\|_F^2 
= \left(1+\frac{\alpha n}{\ell}\right) \sum_{i=r+1}^{\min\{m,n\}} \sigma_i^2(\M{X})$,
and the result follows.
\end{proof}

\subsection{Tensor Bound}
To generalize the result from \cref{thm:matrixbound} to higher dimensions, we first need a result that expresses the error in a Tucker decomposition in terms of the error in each mode.
\begin{lemma}[{\cite[Theorem 5.1]{vann2012new}}]\label{lem:proj}
	Let $\T{X} \in \mb{R}^{n_1 \times \dots \times n_d}$ and let $\M{P}_j \in \mb{R}^{n_j \times n_j}$ for $j = 1, \dots, d$ be a sequence of orthogonal projectors. Then
	\vspace{-.4cm}
	
	$$\| \T{X} - \T{X} \bigtimes_{j=1}^d \M{P}_j \|^2 = \sum_{j=1}^d \| \T{X} \bigtimes_{i=1}^{j-1} \M{P}_i \times_j \left( \M{I} - \M{P}_j \right) \|^2 \leq \sum_{j=1}^d \| \T{X} - \T{X} \times_j \M{P}_j \|^2.$$
\end{lemma}
Recall the notation $n_j^\oslash = \prod_{k \neq j}^d n_k$. 
We present our main error bound result in \cref{thm:tensorbound}, which we frame as the error bound for \cref{alg:rhosvdrekron}. This result can be adapted to also apply to \cref{alg:rsthosvdkron} by following similar techniques to \cite[Theorem 3.2]{minster2020randomized}.
We include the details for completeness in \cref{app:sthosvd_proof}.

\begin{theorem}\label{thm:tensorbound}
	Let $\T{T} = [\T{G}; \M{U}_1,\dots, \M{U}_d]$ be the approximation given by \cref{alg:rhosvdrekron} to $\T{X} \in \mb{R}^{n_1 \times \dots \times n_d}$ with target rank $\V{r} = (r_1,\dots, r_d)$ and oversampling parameter $p$. Let $\ell_j = r_j +p$ for $j = 1,\dots, d$. Then, for sequences $\{ \alpha_j\}_{j=1}^d$ and $\{\beta_j\}_{j=1}^d$ satisfying \cref{eq:alphabeta}, the following bound holds with probability at least $1 - \sum_{j = 1}^d \frac{1}{\beta_j^2}$,
\begin{equation*}
 \| \T{X} - \T{T} \| \leq \left( \sum_{j=1}^d \left(1+ \frac{\alpha_j n_j^\oslash}{\ell_j} \right) \sum_{i=r_j+1}^{n_j} \sigma_i^2 (\M{X}_{(j)}) \right)^{1/2} + \left(\sum_{j=1}^d \sum_{i=r_j+1}^{\ell_j} \sigma_i^2\left(\Mz{X}{j} \right)\right)^{1/2}.
\end{equation*}
\end{theorem}
\begin{proof}
Using \cref{eq:core_svals} to bound $\varepsilon_{\text{core}}$, we need only to bound $\varepsilon_{\text{rand}}$.
From \cref{lem:proj}, we have for $\T[\hat]{T} = [\T[\hat]{G}; \M[\hat]{U}_1,\dots, \M[\hat]{U}_d]$ computed as the intermediate rank-$\V{\ell}$ approximation,
\begin{equation*}
	\| \T{X} - \T[\hat]{T} \|^2 = \| \T{X} - \T{X} \bigtimes_{j=1}^d \M[\hat]{U}_j \M[\hat]{U}_j^\top \|^2 \leq \sum_{j=1}^d \| \T{X} - \T{X} \times_j \M[\hat]{U}_j\M[\hat]{U}_j^\top \|^2 = \sum_{j=1}^d \| \M{X}_{(j)} - \M[\hat]{U}_j\M[\hat]{U}_j^\top \M{X}_{(j)} \|_F^2,
\end{equation*}
where the last equality comes from unfolding the tensor along the $j$-th mode for $j = 1,\dots d$.
We now apply the matrix bound from \cref{thm:matrixbound} on each term in this sum, which gives
	$\| \M{X}_{(j)} - \M[\hat]{U}_j\M[\hat]{U}_j^\top \M{X}_{(j)} \|_F^2 \leq \left(1+ \frac{\alpha_j n_j^\oslash}{\ell_j}   \right) \sum_{i = r_j+1}^{n_j} \sigma_i^2(\M{X}_{(j)})$, 
except with failure probability at most $\frac{1}{\beta_j^2}$. 
Then the failure probability for the entire sum is the union of all $d$ failure probabilities for each mode, which is bounded above by the sum of those probabilities by the union bound. Thus, 
\begin{equation}\label{eq:eps_rand}
	\| \T{X} - \T[\hat]{T} \|^2 \leq \sum_{j=1}^d \| \M{X}_{(j)} - \M[\hat]{U}_j\M[\hat]{U}_j^\top \M{X}_{(j)} \|_F^2 \leq \sum_{j=1}^d \left(1+ \frac{\alpha_j n_j^\oslash}{\ell_j} \right) \sum_{i=r_j+1}^{n_j} \sigma_i^2 (\M{X}_{(j)}), 
\end{equation}
except with probability at most $\sum_{j = 1}^d \frac{1}{\beta_j^2}$.
Then, taking square roots gives $\varepsilon_{\text{rand}}$.

Combining \cref{eq:eps_rand} and \cref{eq:core_svals}, the total error in approximation is 
\begin{equation*}
\varepsilon_{\text{total}} \leq \left( \sum_{j=1}^d \left(1+ \frac{\alpha_j n_j^\oslash}{\ell_j} \right) \sum_{i=r_j+1}^{n_j} \sigma_i^2 (\M{X}_{(j)}) \right)^{1/2} + \left(\sum_{j=1}^d \sum_{i=r_j+1}^{\ell_j} \sigma_i^2\left(\Mz{X}{j} \right)\right)^{1/2},
\vspace{-.1cm}
\end{equation*}
with probability at least $1 - \sum_{j = 1}^d \frac{1}{\beta_j^2}$.
 \end{proof}
 
We consider this result pessimistic due to the factor of $n_j^\oslash$ that comes directly from our analysis of SRHT matrices, but does not appear in our accuracy experiments in \cref{sec:accuracy}. We anticipate that this factor could be improved in future analysis. Also note that the result of \cref{thm:tensorbound} differs from \cite[Theorem 4.2]{che2021efficient} in two main ways: our probability of failure is smaller ($\frac{1}{\beta^2}$ per mode compared to $\frac{1}{\beta}$), and our constant in $\varepsilon_{\text{rand}}$ is smaller as we divide by $\ell_j$ for each $j$. Additionally, we use a different diagonal matrix that is not a Kronecker product.

\section{Parallel Algorithms}\label{sec:parallel}
We design and develop parallel implementations for all the randomized algorithms listed in \cref{tab:compcost}. 
Our implementations are based on TuckerMPI \cite{BKK20}, which uses the \STHOSVD{} algorithm. 
Similarly to TuckerMPI, our implementations leverage distributed-memory clusters to efficiently compute the Tucker decomposition of large multidimensional datasets. 
We employ the following data distribution scheme, proposed in \cite{BKK20}, in our implementations. To distribute a $d$-way input tensor, the processors are organized in a $d$-way processor grid, the dimensions of which are user-determined. 
Each processor owns a subtensor of the input tensor. 
For example, for a $8\times 6 \times 2$ tensor and a $2\times 3\times 1$ processor grid, each processor owns a $4\times 2 \times 2$ subtensor. 
All matrices involved in our algorithms are stored redundantly by every processor.

In the following sections we describe two optimizations for computing the sketch tensor via multi-TTMs as well as parallel implementations of the two algorithms from \cref{sec:seq_algs} with the smallest computational cost, \cref{alg:rsthosvdkron} and \cref{alg:rhosvdrekron}. 
We also compare our implementations with previous work from \cite{choi2018high}.   

\subsection{All-at-Once Multi-TTM}\label{sec:aao-mttm}
The multi-TTM operation is one of the most expensive kernels of our algorithms and appears twice; we compute a multi-TTM both to form the sketch tensor $\T{Y}$ and to form the core tensor $\hat{\T{G}}$ (e.g. in \cref{line:sketchMultiTTM} of \cref{alg:rhosvdrekron}). 
It is thus crucial to optimize this operation. 

A parallel implementation of a single TTM is proposed in \cite{BKK20}. One way to implement the multi-TTM is to simply perform this existing TTM algorithm multiple times, as shown in \cref{alg:is-ttm}; we call this approach the in-sequence multi-TTM or IS-mTTM. Note that in the algorithms bars over letters denote local data. In this in-sequence approach, a reduce-scatter is performed at the end of each TTM
operation, reducing the amount of data each processor owns so that the computation cost in the next TTM is also reduced. Generally, this approach performs additional communications to obtain lower computational cost. However, depending on the size of the local tensor, the reduction in computational cost may not justify the increased communication cost.

Our algorithm is shown in \cref{alg:aao-ttm}, which we call the all-at-once multi-TTM or 
\vspace{-.3cm}

\hspace*{-\parindent}%
\begin{minipage}[t]{.49\textwidth}
  \begin{algorithm}[H]
    \caption{In-sequence multi-TTM}\label{alg:is-ttm}
    \begin{algorithmic}[1]\footnotesize
      \Function{$\T[\bar]{Y}=$ IS-mTTM}{$\bar{\T{X}}, j, \{\M{M}_j\}, \T{P}$} 
        \State $(p_1,\,\dots,\, p_d) = \text{procID}(\T{P})$
        \State $\bar{\T{T}} = \bar{\T{X}}$
        \For{$i = 1:d$ and $i\neq j$}
          \State $\T{F}=\T{P}(p_1,\dots,p_{i-1},:, p_{i+1},\dots, p_d)$
          \State $\bar{\M{T}}_{(i)} = \M[\bar]{M}_i\bar{\M{T}}_{(i)}$ 
          \State $\bar{\M{Y}}_{(i)}$=\Call{Reduce-Scatter}{$\bar{\M{T}}_{(i)},\,\T{F}$}
        \EndFor
      \EndFunction
    \end{algorithmic}  
  \end{algorithm}
\end{minipage}
\hfill
\begin{minipage}[t]{.49\textwidth}
  \begin{algorithm}[H]
    \caption{All-at-once multi-TTM}\label{alg:aao-ttm}
    \begin{algorithmic}[1]\footnotesize
      \Function{$\T[\bar]{Y}=$ AAO-mTTM}{$\bar{\T{X}}, j, \{\M{M}_j\}, \T{P}$} 
        \State $(p_1,\,\dots,\, p_d) = \text{procID}(\T{P})$
        \State $\bar{\T{T}} = \bar{\T{X}}$
        \For{$i = 1:d$ and $i\neq j$}
          \State $\bar{\M{T}}_{(i)} = \M[\bar]{M}_i \bar{\M{T}}_{(i)} $
        \EndFor
					\State $\T{S}=\T{P}(:,\,\dots,\,:,\,p_j,\,:,\, \dots,\, :)$
					\State $\bar{\M{Y}}_{(j)}$ = \Call{Reduce-Scatter}{$\bar{\M{T}}_{(j)},\T{S}$}
      \EndFunction
    \end{algorithmic}    
  \end{algorithm}  
  \vspace{.1cm}
\end{minipage}

\noindent AAO-mTTM. In this approach, we avoid communication until all matrices have been multiplied with the local tensor. 
This strategy reduces communication by increasing the cost of storage and computation. 

The most significant difference between \cref{alg:is-ttm} and \cref{alg:aao-ttm} is that at the end of any iteration $i$ of the in-sequence approach, we form the intermediate result $\T{Y} = \T{X} \times_1 \M{M}_1 \times \dots \times_i \M{M}_i$; in the all-at-once approach, each processor stores a contribution to $\T{Y}$ until all iterations are completed and the all-reduce at the end of the algorithm forms the final result.   

\subsubsection{Cost Analysis}
To simplify the notation, we assume that the input tensor is a $d$-way cubic tensor $\T{X} \in \mb{R}^{n \times \dots \times n}$, that the processor tensor is size $q$ in each mode ($q^d = P$ processors in total), and that the input matrices $\{\M{M}_i\}_{i=1}^d$ are of the same size $s\times n$ with $s < n$.   
With this notation, we analyze the per-processor compuation and communication costs of performing the multi-TTM $\T{X} \times_1 \M{M}_1 \times \dots \times_{j-1} \M{M}_{j-1}\times_{j+1} \M{M}_{j+1} \times \dots  \times_d \M{M}_d$ using \Cref{alg:aao-ttm} and compare it with that of using \Cref{alg:is-ttm}.

The computational cost of \cref{alg:is-ttm} can be written as
\begin{equation}\label{eq:isTTM-comp-cost}
  C_{\text{in-sequence}} = 2\left(\frac{sn^d}{q^d} + \frac{s^2n^{d-1}}{q^d} + ... + \frac{s^{d} n}{q^d}\right) = 2\left(\sum_{i=1}^{d} \frac{s^in^{d-i+1}}{q^d}\right)\,,
\end{equation}
and the computational cost of \cref{alg:aao-ttm} can be written as
\begin{equation}\label{eq:aaoTTM-comp-cost}
  C_{\text{all-at-once}} =  2\left(\frac{sn^d}{q^d} + \frac{s^2n^{d-1}}{q^{d-1}} + ... + \frac{s^{d} n}{q^2}\right) = 2\left(\sum_{i=1}^{d} \frac{s^in^{d-i+1}}{q^{d-i+1}}\right)\,.
\end{equation} 

The $i$th terms of the summations in both \cref{eq:aaoTTM-comp-cost} and \cref{eq:isTTM-comp-cost} represent the cost of multiplying the local factor matrix $\M[\bar]{M}_i$ of size $s\times \frac{n}{q}$ and the $i$th mode unfolding of tensor $\bar{\T{Y}} = \bar{\T{X}} \times_1 \M[\bar]{M}_1 \times \dots \times_{i-1} \M[\bar]{M}_{i-1}$. 
In \cref{alg:is-ttm}, due to the reduce-scatter at each iteration, $\bar{\M{Y}}_{(i)}$ is of size $\frac{n}{q} \times \frac{s^{i-1}n^{d-i}}{q^{d-1}}$. In \cref{alg:aao-ttm}, $\M[\bar]{M}_i$ is of the same size. 
However, since the reduction is delayed until the last mode, $\bar{\M{Y}}_{(i)}$ is of size $\frac{n}{q} \times \frac{s^{i-1}n^{d-i}}{q^{d-i}}$.
Comparing \cref{eq:isTTM-comp-cost} and \cref{eq:aaoTTM-comp-cost}, it is easy to see that the computational cost of an \aaoTTM{} is always higher than that of an \isTTM{} because each of the summands in \cref{eq:aaoTTM-comp-cost} is at least as large as the corresponding term in \cref{eq:isTTM-comp-cost}. This increase can be small, however, in certain cases. Note the two series of summands are geometric and have the same leading term. For \cref{eq:isTTM-comp-cost}, the ratio of the series is $\frac{s}{n}$ while the ratio of the series in \cref{eq:aaoTTM-comp-cost} is $\frac{sq}{n}$. As the subranks become smaller compared to the tensor dimensions (i.e. $s \ll n$), the sums of the two series get closer to their first terms and thus their difference becomes smaller. 

Using the same notations and the $\alpha$-$\beta$-$\gamma$ model \cite{Chan2007}, we can express the communication cost of \cref{alg:is-ttm} and \cref{alg:aao-ttm}. In \cref{alg:aao-ttm}, there is only one communication step at the end where all $P$ processors communicate their local tensor $\bar{\T{T}}$. $\bar{\T{T}}$ is a $d$-way tensor with size $s$ for all of its modes except for the $j$th mode which has size $\frac{n}{q}$. Therefore, the communication cost for each processor is:
  $\alpha \mathcal{O}(\log P) + \beta \mathcal{O}\left(\frac{s^{d-1}n}{q}\right)$.
The communication cost of \cref{alg:is-ttm} is more complicated because there are $d-1$ communication steps and each one involves $\bar{\T{T}}_{(i)}$ which is changing in size. At the $i$th iteration of the for loop, $\bar{\T{T}}_{(i)}$ has size $s \times \frac{s^{i-1}n^{d-i}}{q^{d-1}}$. Therefore, the communication cost can be written as
	$\alpha\mathcal{O}( d \log q) + \beta \mathcal{O}\left(\sum_{i=1}^{d-1}\frac{s^in^{d-i}}{q^{d-1}} \right) = \alpha\mathcal{O}(\log P) + \beta \mathcal{O}\left(\frac{sn^{d-1}}{q^{d-1}}\right)$.
We obtain the right hand side of the equation using the assumption that $s \ll n$ so that the summation is approximated by its first summand. 
When $sq < n$, the communication cost of \aaoTTM{} is smaller. 
As the ratio $\frac{n}{sq}$ increases, the benefits of using \aaoTTM{} become more substantial. 

\subsection{Dimension Tree Optimization}
In \cref{sec:dimension_tree}, we discuss how dimension trees can be used to make the randomized sketches less expensive for \cref{alg:rhosvdrekron}. 
In \cref{alg:dimTreeWithAaottm}, we present an implementation using AAO-mTTM. 
Here, $\V{m}$ and $\V{n}$ are sets of integers in the range $[1,d]$. 
This algorithm returns $\T{Y}^{(j)}$, the sketch of $\Mz{X}{j}$ in tensor format, for all integers $j\in [1,d]$.
Since dimension trees are a tool for reusing only local intermediate results, \cref{alg:dimTreeWithAaottm} can also be modified to use IS-mTTM. 
\begin{algorithm}[!ht]
  \begin{algorithmic}[1]\footnotesize
  \caption{All modes multi-TTM using the dimension tree optimization and \aaoTTM{}. }\label{alg:dimTreeWithAaottm}
  \Function{$\{\T{Y}^{(j)}\}=$All-Modes-Multi-TTM}{$\bar{\T{X}}$, $\{\M[\bar]{\Phi}_j\}$,\, $\V{m}$, $\V{n}$, $\T{P}$}
    \State $(p_1,\,\dots,\, p_d) = \text{procID}(\T{P})$
    \State $\bar{\T{Y}} = \bar{\T{X}}$
    \For{$i \in \V{n}$ }
      \State $\bar{\M{Y}}_{(i)} = \bar{\M{\Phi}}_i\bar{\M{Y}}_{(i)} $
    \EndFor
    \If{$\V{m}$ only contains 1 integer, $j$}
      \State $\T{S}=\T{P}(:,\,\dots,\,p_j,\, \dots,\, :)$
      \State $\bar{\T{Y}}^{(j)}$ = \Call{Reduce-Scatter}{$\bar{\T{Y}},\T{S}$}
    \Else
      \State split $\V{m}$ in equal half $\V{m}_1$ and $\V{m}_2$
      \State \Call{All-Modes-Multi-TTM}{$\bar{\T{Y}}$, $\{\bar{\M{\Phi}}_j\}$, $\V{m}_1$, $\V{m}_2$, $\T{P}$}
      \State \Call{All-Modes-Multi-TTM}{$\bar{\T{Y}}$, $\{\bar{\M{\Phi}}_j\}$, $\V{m}_2$, $\V{m}_1$, $\T{P}$}
    \EndIf
  \EndFunction
  \end{algorithmic}
\end{algorithm}

\subsection{Principal Algorithms}
We now combine the discussed multi-TTM approaches and the dimension tree optimization within parallel implementations of our two best algorithms, \cref{alg:rsthosvdkron,alg:rhosvdrekron}. The pseudocode is provided in \cref{alg:para_rsthosvdkron,alg:para_rhosvdkronreuse}, respectively.  Note that for both of these algorithms, when the size of core $\hat{\T{G}}$ is large, it is better to perform the \isTTM{} as described in \cref{alg:is-ttm} to produce a distributed $\hat{\T{G}}$ so that the parallel STHOSVD can be used to reduce cost. When $\hat{\T{G}}$ is very small, performing STHOSVD in parallel can be counterproductive as the communication cost will dominate; it is better in this case to perform an all-gather among all processors after the \isTTM{}, producing $\hat{\T{G}}$ redundantly on every processor and then use the sequential STHOSVD.

\begin{algorithm}[!ht]
  \begin{algorithmic}[1]\footnotesize
		\caption{Parallel algorithm for \cref{alg:rsthosvdkron}}
		\label{alg:para_rsthosvdkron}
    \Function{\rSTHOSVDkron}{$\T{X}$, $\V{r}$, $p$, $\T{P}$} \Comment{$\T{X}$ is distributed, the local part of $\T{X}$ is denoted as $\bar{\T{X}}$}
        \State $\hat{\T{G}} = \T{X}$
        \State Redundantly compute matrix of subranks $\M{S}$ 
        \For{$i$ = 1:d}
        \For{$j$ = 1:d and $j\neq i$}
            \State Redundantly draw $d-1$ random matrices $\M{\Phi}_{i,j}\in \mb{R}^{n_j \times s_{i,j}}$
        \EndFor
        \State $\T[\bar]{Y} \leftarrow  \Call{AAO-mTTM}{\hat{\T{G}},\, i,\, \{\M{\Phi}_{i,1} \dots \M{\Phi}_{i,d}\},\, \T{P}}$ \Comment{Can also use \isTTM{}}
        \State $\T{Y} = \Call{All-Gather}{\T[\bar]{Y},\T{P}}$
        \State $\M[\hat]{U}_i = $ \Call{QR}{$\Mz{Y}{i}$} \Comment{\Call{QR}{} is serial as every processor owns global $\T{Y}$}
        \State $\hat{\T{G}} \leftarrow  \Call{TTM}{\hat{\T{G}},\, \M[\hat]{U}_{i}^{\Tra},\,i}$ \Comment{For a single TTM we use the implementation proposed in \cite{BKK20}} \label{line:para_rsthosvdkron:applyFactorMatrices}
        \EndFor
        \State $[\T{G}, \{\M{V}_i\}] = $ STHOSVD$(\hat{\T{G}}, \V{r})$ 
        \For{$i = 1,\, \dots,\, d$}
          \State $\M{U}_i = \hat{\M{U}}_i \M{V}_i$ \Comment{Computed with local matrix multiplication}
        \EndFor
    \EndFunction
  \end{algorithmic}
\end{algorithm}

\begin{algorithm}[!ht]
	\caption{Parallel algorithm for \cref{alg:rhosvdrekron} with \aaoTTM{} and dimension trees}
	\label{alg:para_rhosvdkronreuse}
	\begin{algorithmic}[1]\footnotesize
		\Function{\rHOSVDkronreuse{}}{$\T{X}$, $\V{r}$, $p$, $\T{P}$}
				\State Compute subranks $\V{s}$
				\For{$i$ = 1:$d$}
					\State Redundantly draw $d-1$ random matrices $\M{\Phi}_{i}\in \mathbb{R}^{s_i\times n_i}$
				\EndFor
				\State $\{\T{Y}^{(1)}, \dots, \T{Y}^{(d)}\} = \Call{All-Modes-Multi-TTM}{\bar{\T{X}},  \{\M{\Phi}_1, \dots, \M{\Phi}_d\}, \{1,\dots, d\}, \emptyset, \T{P}}$
				\For{$i$ = 1:$d$}
					\State $\hat{\M{U}}_i = $\Call{QR}{$\M{Y}^{(i)}_{(i)}$} \Comment{Serial QR decomposition of the mode $i$ unfolding of $\T{Y}^{(i)}$}
				\EndFor
				\State $\hat{\T{G}} = \Call{IS-mTTM}{\bar{\T{X}}, \emptyset, \{\hat{\M{U}}_1^{\Tra}, \dots, \hat{\M{U}}_d^{\Tra}\}, \T{P}}$
				\State $[\T{G}, \{\M{V}_j\}] = $ STHOSVD$(\hat{\T{G}}, \V{r})$ 
			\For{$i = 1,\, \dots,\, d$}
				\State $\M{U}_i = \hat{\M{U}}_i \M{V}_i$ \Comment{Computed with local matrix multiplication}
			\EndFor
		\EndFunction
	\end{algorithmic}
\end{algorithm}

\subsection{Comparison to Previous Work}\label{ssec:comparison}
We compare our parallel algorithms with previous approaches, namely the parallel STHOSVD algorithm in \cite{BKK20} and the approach from \cite{choi2018high}. Assume that the input tensor is a $d$-way tensor $\T{X} \in \mb{R}^{n \times \dots \times n}$ with rank $(r,r,\dots, r)$ and that each mode of the $d$-way processor tensor has size $q$. 
We also assume that $s < r < l \ll n$ where $s$ is the subrank for each mode and $\ell = r+p$ with $p$ the oversampling parameter. 
Here $s=\ell^{1/(d-1)}\approx r^{1/(d-1)}$.
In \cite{choi2018high}, Choi et al. proposed a data distribution scheme and a tensor matricization strategy that reduces the communication costs of the Gram and TTM kernels. 
More specifically, in this new method, before the Gram computation, communication is performed every other mode to redistribute the tensor unfolding from a 2D distribution to a block-column 1D distribution, which can avoid the communication cost of the later TTM operations required to form the core tensor. 
This method is shown to achieve speedup over the parallel implementation proposed in \cite{BKK20}. 
However, these optimizations are not suitable for our randomized algorithms for the following reason. Instead of computing the Gram matrix of each mode unfolding, we compute each mode sketch with a multi-TTM. 
Since each sketch is much smaller at the end of the multi-TTM, it is beneficial to delay communicating the sketches as much as possible before all multi-TTM's are completed. 
However, the redistribution proposed by \cite{choi2018high} requires every processor to communicate the entire uncompressed local tensor before the computation. 
While \cite{choi2018high} also uses randomization to reduce computation, their focus is on reducing the cost of computing the eigenvectors of the Gram matrix, which is achieved by using a modified randomized SVD to replace the eigendecomposition. 

To illustrate the benefit of using Kronecker-structured random matrices, we also compare \cref{alg:para_rsthosvdkron} and \cref{alg:para_rhosvdkronreuse} with the parallel version of \cref{alg:rsthosvd}, which uses dense Gaussian random matrices. The parallelized \cref{alg:rsthosvd} is very similar to \cref{alg:para_rsthosvdkron} with the only difference being that we use the parallel multi-TTM to apply the Kronecker-structured random matrices to the input tensor while in parallelized \cref{alg:rsthosvd} we use a single parallel TTM operation to apply each of the dense Gaussian random matrices.

\begin{table}[h]
	\centering \footnotesize
	{\renewcommand\arraystretch{1.6}
	\setlength\tabcolsep{3pt}
\begin{tabular}{c|c|c|c|c}
\hline
\textbf{Algorithm} & \begin{tabular}[c]{@{}c@{}}\textbf{Form} $\{\M{U}\}$\\ \textbf{comp cost}\end{tabular} & \begin{tabular}[c]{@{}c@{}}\textbf{Form} $\{\M{U}\}$\\ \textbf{comm cost}\end{tabular} & \multicolumn{1}{c|}{\begin{tabular}[c]{@{}l@{}}\textbf{Form} $\T{G}$\\ \textbf{comp cost}\end{tabular}} & \begin{tabular}[c]{@{}l@{}}\textbf{Form} $\T{G}$\\ \textbf{comm cost}\end{tabular} \\ \hline
STHOSVD \cite{BKK20} 
& $\frac{n^{d+1}}{P}$
& $\alpha\mathcal{O}(dP^{1/d}) + \beta\mathcal{O}(\frac{n^d}{P})$
& $2\frac{rn^d}{P}$
& $\alpha\mathcal{O}(\log P) + \beta \mathcal{O}(\frac{rn^{d-1}}{P^{1-1/d}})$  \\\hline
\cite{choi2018high} 
& $\frac{n^{d+1}}{P}$
& $\alpha\mathcal{O}(dP) + \beta\mathcal{O}(\frac{n^d}{P})$ 
& $2\frac{rn^d}{P}$
& -  \\ \hline
\begin{tabular}[c]{@{}l@{}}\cref{alg:para_rsthosvdkron}\end{tabular} 
& $2\frac{r^{1/(d-1)}n^d}{P}$     
& $\alpha\mathcal{O}(d\log P) + \beta\mathcal{O}(\frac{drn}{P^{1/d}})$  
& $2\frac{rn^d}{P}$         
& $\alpha\mathcal{O}(\log P) + \beta \mathcal{O}(\frac{rn^{d-1}}{P^{1-1/d}})$  \\ \hline
\begin{tabular}[c]{@{}l@{}}\cref{alg:para_rhosvdkronreuse}\end{tabular} 
& $4\frac{r^{1/(d-1)}n^d}{P}$
& $\alpha\mathcal{O}(d\log P) + \beta\mathcal{O}(\frac{drn}{P^{1/d}})$ 
& $2\frac{rn^d}{P}$         
& $\alpha\mathcal{O}(\log P) + \beta \mathcal{O}(\frac{rn^{d-1}}{P^{1-1/d}})$  \\ \hline
\end{tabular}}
\caption{Comparison of computation and communication cost per processor. More details can be found in \cref{sec:sm_cost}. }\label{tab:parallel_cost_comp}
\end{table}

\section{Experimental Results}\label{sec:experiments}
We now demonstrate the numerical benefits of our algorithms by considering the accuracy of our sequential algorithms in \cref{sec:accuracy}, the performance of the multi-TTM and dimension tree optimizations in \cref{sec:ismttm_vs_aaomttm,sec:dimTree_results}, respectively, and the performance of the parallel implementations of our randomized algorithms on synthetic data in \cref{sec:strong_scaling} and on two real datasets in \cref{sec:miranda,sec:sp}.

\paragraph*{Computing platform}
The results shown in \cref{sec:accuracy} are generated by running MATLAB implementations of the sequential algorithms on a single node server. The experiments shown from \cref{sec:ismttm_vs_aaomttm} to \cref{sec:sp} are run on the Andes cluster at Oak Ridge Leadership Computing Facility. 
The system consists of 704 compute nodes with 2 AMD EPYC 7302 16-core CPU's and 256 GB of RAM. 
We directly call OpenBLAS and the Netlib implementation of LAPACK for local linear algebra kernels, which are the only available libraries on Andes. 

\subsection{Accuracy Results}\label{sec:accuracy}
We present an experiment on a synthetic tensor that demonstrates the accuracy of our algorithms compared to existing deterministic and randomized algorithms. In this experiment, we use both SRHT and Gaussian random matrices to compare numerical accuracy to the theoretical results we derived in \cref{sec:erroranalysis}.

We construct a synthetic 3-way tensor $\T{X} \in \mb{R}^{500 \times 500 \times 500}$ by forming a (super-) diagonal tensor with geometrically decreasing entries and multiplying that tensor by a random orthogonal matrix along each mode. 
We set the largest entry of the original tensor to be 1 and choose the rate at which the core entries decrease to be 0.4 so that the 40th entry is approximately machine precision.
In our experiment, we compress this tensor to rank $(10,10,10)$ using an oversampling parameter of $p = 5$. 
We show boxplots of the relative error over 100 trials in \cref{fig:boxplot}, comparing results from all our algorithms (\cref{alg:rhosvd,alg:rsthosvd,alg:rsthosvdkron,alg:rhosvdrekron}) using Gaussian random matrices as well as our principal algorithms (\cref{alg:rsthosvdkron,alg:rhosvdrekron}) using SRHT random matrices. We also compare our truncation strategy to the strategy in \cite{che2021efficient}, and compare the relative error from all randomized algorithms to the relative error obtained from deterministic STHOSVD, \cref{alg:sthosvd}.

In \cref{fig:boxplot}, we see that the relative errors for all our randomized algorithms deviate from the deterministic relative error by at most 10\% for the given rank (the medians are within 1\%)
Also, the relative errors for each trial are very close together with each algorithm, as even the outliers are within the same order of magnitude, with a standard deviation of $1.9\times 10^{-5}$.
Regardless of random matrix distribution or whether we reuse $\M{\Phi}_j$ matrices or generate new matrices for each mode, we do not lose significant accuracy compared to either the deterministic or standard randomized approaches. We can also see that the truncation strategy we employ in \cref{alg:rhosvd,alg:rsthosvd,alg:rsthosvdkron,alg:rhosvdrekron} is much more consistently accurate than the strategy in \cite{che2021efficient}.

\begin{figure}[!ht]
\centering
\tikzexternaldisable
\tikzsetnextfilename{relative_err_gaus_box_plot}
\begin{subfigure}{.32\textwidth}
	\centering
\resizebox{.91\textwidth}{!}{
\begin{tikzpicture}
  \begin{axis}
    [
    title = {Gaussian Random Matrices},  
    title style={at={(.5,1.1)},anchor=north},
    boxplot/draw direction=y,
    xtick={1,2,3,4},
    xticklabels={\cref{alg:rhosvd}, \cref{alg:rsthosvd},\cref{alg:rsthosvdkron}, \cref{alg:rhosvdrekron}},
    xticklabel style={rotate=45, anchor=east},
    xmin=0,xmax=5,
    legend style={at={(1,.98)},anchor=north east},
    ymin=.0001,ymax=.00012
    ]
    \addplot[dashed,color=black] coordinates{(0,0.000104858)(5,0.000104858)};
    \addlegendentry{\STHOSVD};
    \addplot+[
    boxplot prepared={
      median=0.000105385,
      upper quartile=0.000105481,
      lower quartile=0.00010515,
      upper whisker=0.000105883,
      lower whisker=0.000104927
    },
    ] coordinates {(0,0.000115701)(0,0.000110455)(0,0.000107454)(0,0.000107085)(0,0.000106851)(0,0.000106615)(0,0.000106476)(0,0.00010618)(0,0.000106165)(0,0.000106126)(0,0.000106109)};
    \addplot+[
    boxplot prepared={
      median=0.000105464,
      upper quartile=0.000106586,
      lower quartile=0.000105142,
      upper whisker=0.000108635,
      lower whisker=0.000104954
    },
    ] coordinates {(0,0.000109628)(0,0.0001097)(0,0.000110248)(0,0.000110396)(0,0.000110454)(0,0.000110455)(0,0.000111067)(0,0.000111545)(0,0.0001153)(0,0.000188826)};
    \addplot+[
    boxplot prepared={
      median=0.000105339,
      upper quartile=0.000106089,
      lower quartile=0.000105143,
      upper whisker=0.000107096,
      lower whisker=0.000104927
    },
    ] coordinates {(0,0.000107898)(0,0.000108203)(0,0.000108592)(0,0.000115701)(0,0.000136605)};
    \addplot+[
    boxplot prepared={
      median=0.000104994,
      upper quartile=0.000105205,
      lower quartile=0.000104918,
      upper whisker=0.000105562,
      lower whisker=0.00010487
    },
    ] coordinates {(0,0.000107696)(0,0.000106753)(0,0.000106538)(0,0.000106345)(0,0.000106236)(0,0.000106145)(0,0.000105954)(0,0.000105645)};
   
  \end{axis};

\end{tikzpicture}}
	\vspace{.1cm}
\end{subfigure}
\tikzexternaldisable
\tikzsetnextfilename{relative_err_srht_box_plot}
\begin{subfigure}{.32\textwidth}
	\centering

\resizebox{.89\textwidth}{!}{
\begin{tikzpicture}
  \begin{axis}
    [
    title = {SRHT Random Matrices},  
    title style={at={(.5,1.1)},anchor=north},
    boxplot/draw direction=y,
    xtick={1,2},
    xticklabels={\cref{alg:rsthosvdkron},\cref{alg:rhosvdrekron}},
    xticklabel style={rotate=45, anchor=east},
    xmin=0,xmax=3,
    ymin=.0001,ymax=.00012
    ]
    \addplot[dashed,color=black] coordinates{(0,0.000104858)(9,0.000104858)};

    \addplot+[
    boxplot prepared={
      median=0.00010488,
      upper quartile=0.00010492,
      lower quartile=0.00010487,
      upper whisker=0.00010497,
      lower whisker=0.00010486
    },
    ] coordinates {(0,0.00010501)(0,0.00010504)(0,0.00010506)(0,0.00010506)(0,0.00010509)(0,0.00010513)(0,0.00010517)(0,0.00010524)(0,0.00010538)};
    \addplot+[
    boxplot prepared={
      median=0.00010503,
      upper quartile=0.000105173,
      lower quartile=0.00010495,
      upper whisker=0.0001054,
      lower whisker=0.00010487
    },
    ] coordinates {(0,0.00010554)(0,0.00010556)(0,0.00010558)(0,0.00010565)(0,0.00010585)(0,0.00010587)(0,0.0001059)(0,0.00010599)(0,0.00010628)(0,0.00010654)(0,0.00010712)(0,0.00010901)(0,0.00010922)};
    
  \end{axis};

\end{tikzpicture}}
	\vspace{.1cm}
\end{subfigure}
\tikzexternaldisable
\tikzsetnextfilename{relative_err_comp_boxplot}
\begin{subfigure}{.32\textwidth}
	\centering

\resizebox{.85\textwidth}{!}{
\begin{tikzpicture}
  \begin{axis}
    [
    title = {Truncation Comparison (Gaussian)},  
    title style={at={(.6,1.1)},anchor=north},
    boxplot/draw direction=y,
    xtick={1,2},
    xticklabels={\cref{alg:rsthosvdkron}, \cite{che2021efficient} trunc.},
    xticklabel style={rotate=45, anchor=east},
    xmin=0,xmax=3,
    ]
    \addplot[dashed,color=black] coordinates{(0,0.000104858)(3,0.000104858)};

    \addplot+[
    boxplot prepared={
      median=0.000105339,
      upper quartile=0.000106089,
      lower quartile=0.000105143,
      upper whisker=0.000107096,
      lower whisker=0.000104927
    },
    ] coordinates {(0,0.000107898)(0,0.000108203)(0,0.000108592)(0,0.000115701)(0,0.000136605)};
    
    \addplot+[
    boxplot prepared={
      median=0.000682941,
      upper quartile=0.001125117,
      lower quartile=0.000424934,
      upper whisker=0.00211586,
      lower whisker=0.000146188
    },
    ] coordinates {(0,0.0021844)(0,0.002413532)(0,0.002717415)(0,0.002944279)(0,0.003034444)(0,0.003100188)(0,0.003135402)(0,0.003671784)(0,0.004052918)};
    
  \end{axis};

\end{tikzpicture}}
\end{subfigure}
\caption{Boxplots of relative errors for our randomized algorithms using Gaussian random matrices (left) and SRHT random matrices (center) compared to the relative error for STHOSVD for the synthetic tensor with geometrically decreasing values. We also compare the truncation methods we use to those used in \cite{che2021efficient} (right).}
\label{fig:boxplot}
\end{figure}
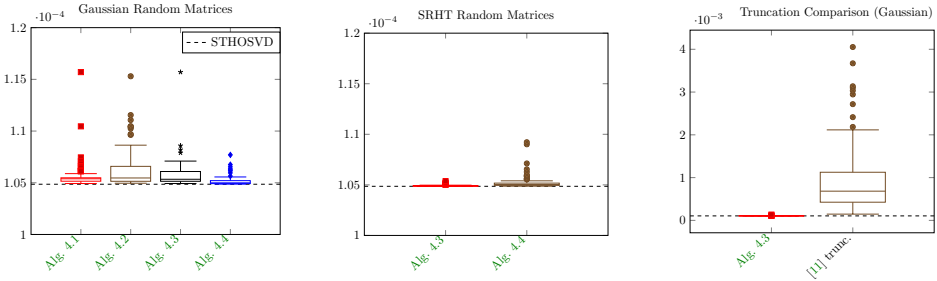

We show an additional accuracy experiment in \cref{sec:sm_accuracy}. 
Overall, we see comparable relative error for \cref{alg:rsthosvdkron,alg:rhosvdrekron} to both deterministic STHOSVD and existing randomized approaches.
While our theoretical results (\cref{thm:tensorbound}) apply only to SRHT matrices, these experimental results suggest that Gaussian matrices perform comparably in terms of accuracy.

\subsection{In-Sequence TTM vs All-At-Once TTM}\label{sec:ismttm_vs_aaomttm}
To compare the performance of the two TTM approaches we discuss in \cref{sec:aao-mttm}, we conduct an experiment performing an in-sequence multi-TTM and an all-at-once multi-TTM of a 3-way tensor $\T{X} \in \mb{R}^{800 \times 800 \times 800}$ with matrices $\M{U}, \M{V}, \M{W} \in \mb{R}^{s \times 800}$, with varying $s$. 
For this experiment, we use 64 cores (2 nodes) arranged in a $4\times4\times4$ processor grid.
The results of this experiment are shown in \cref{fig:ttmComp}, where we can see that when the matrices have relatively few rows (when $sq \ll n$), the all-at-once multi-TTM is much more communication-efficient. 
We observe speedups ranging from 27\% to over $2\times$. 
This fits our prediction in \cref{sec:aao-mttm}. 
However, when $\frac{n}{sq}$ is very large (i.e. the $s=5$ case), both algorithms are cheap and communication is not as dominant. 
When $\frac{n}{sq}$ is small, all-at-once multi-TTM has increased computational cost, and also loses its advantage in communication cost. 
In this case, \isTTM{} is preferred to \aaoTTM.
For this reason, we think the all-at-once optimization is more suitable for the sketching phase where the random matrices tend to have very few rows.

Note that the gray bars in \cref{fig:ttmComp} represent overhead cost. 
For \isTTM{}, this overhead mainly comes from reorganizing the data in memory before and after communication (MPI collectives). 
Since \aaoTTM{} avoids those communication steps, it also avoids those reorganizing costs.
For higher dimensions, the benefits of \aaoTTM{} still depend on $\frac{n}{sq}$ being large. 
If this ratio is fixed, \aaoTTM{} will continue to outperform \isTTM{}.

\subsection{Dimension Tree Optimization}\label{sec:dimTree_results}
To demonstrate the benefits of using dimension trees, we run \cref{alg:rhosvdrekron} with and without dimension trees on synthetic tensors with an increasing number of modes such that the total size of the input tensor and its rank are kept close to constant. 
We benchmark the time it takes for both methods to apply the random matrices $\{\Phi\}$ to the input tensor and present the results in \cref{fig:dimTree_Comp}.
This computation corresponds to \cref{line:sketchMultiTTM} of \cref{alg:rhosvdrekron}. 
Since the communication and overhead costs are low, we see that the practical speedup from using dimension trees aligns closely with the theoretical prediction, with a computational reduction of $d/2$ as described in \cref{sec:dimension_tree}.

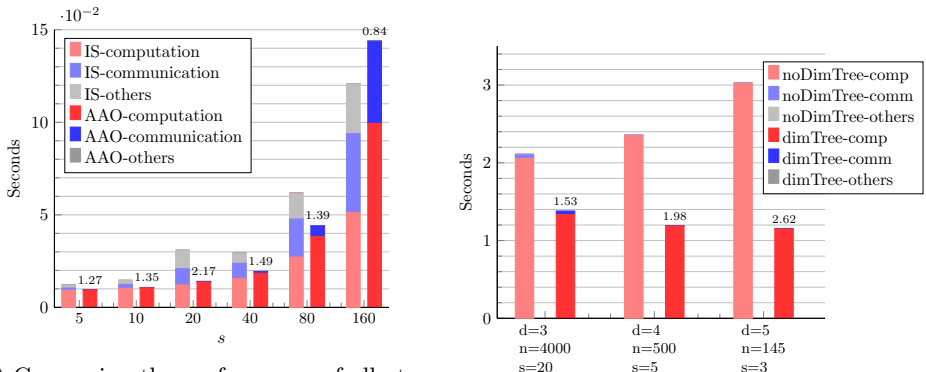
\begin{figure}[!ht]
	\renewcommand{\datafile}{data/TTM_combined.dat}
	\tikzexternaldisable
	\tikzsetnextfilename{TTM_comp}
	\begin{subfigure}{.45\textwidth}
		\centering
\resizebox{.9\textwidth}{!}{
\begin{tikzpicture}
  \begin{axis}[
    axis lines*=left,
    ymajorgrids,
    yminorgrids,
    minor tick num=4,
    ymin=0, 
    ymax= 0.15,
    scaled y ticks=base 10:2,
    ylabel = {Seconds},
    ylabel near ticks,
    ybar stacked,
    bar width=8pt,
    xtick=data,
    xmin=-1,
    xmax=22.5,
    xtick={0.75, 4.75, 8.75, 12.75, 16.75, 20.75},
    xticklabels ={
      5, 10, 20, 40, 80, 160, 320
    },
    xticklabel style={font=\small,anchor= north},
    xlabel = {$s$},
    cycle list={
      {fill=red!80, draw=red!80},
      {fill=blue!80, draw=blue!80},
      {fill=gray!80, draw=gray!80}
    },
    visualization depends on=rawy \as \rawy,
    point meta={(\thisrow{IS-computation}+\thisrow{IS-communication}+\thisrow{IS-others})/(\thisrow{AAO-computation}+\thisrow{AAO-communication}+\thisrow{AAO-others}) },
    nodes near coords style={
        black,
        font=\scriptsize,
        shift={(axis direction cs:0,\rawy/2)},
        anchor=south,
    },
  ]
    \addplot table [x=AAO-ind, y=AAO-computation]   {\datafile};
    \addplot table [x=AAO-ind, y=AAO-communication] {\datafile};
    \addplot+ [nodes near coords] table [x=AAO-ind, y=AAO-others] {\datafile};
  \end{axis}

  \begin{axis}[
    axis lines*=left,
    ymin=0, 
    ymax= 0.15,
    scaled y ticks=base 10:2,
    ybar stacked,
    bar width=8pt,
    xmin=-1,
    xmax=22.5,
    ticks=none,
    legend cell align={left},
    legend entries={IS-computation, IS-communication, IS-others, AAO-computation, AAO-communication, AAO-others},
    legend pos= north west, cycle list={
      {fill=red!50, draw=red!50},
      {fill=blue!50, draw=blue!50},
      {fill=gray!50, draw=gray!50},
      {fill=red!80, draw=red!80},
      {fill=blue!80, draw=blue!80},
      {fill=gray!80, draw=gray!80}
    }
  ]
    \addplot table [x=IS-ind, y=IS-computation]   {\datafile};
    \addplot table [x=IS-ind, y=IS-communication] {\datafile};
    \addplot table [x=IS-ind, y=IS-others]        {\datafile};
    \addplot table [x=IS-ind, y=fake]        {\datafile};
    \addplot table [x=IS-ind, y=fake]        {\datafile};
    \addplot table [x=IS-ind, y=fake]        {\datafile};
  \end{axis}
\end{tikzpicture}}

		\caption{Comparing the performance of \aaoTTM{} and \isTTM{} by multiplying a $800\times 800 \times 800$ synthetic tensor with an $s\times 800$ matrix on all modes ($s\in [5,160]$). The processor grid used is $4\times 4\times 4$. The labels on top of bars indicate the overall speedup achieved by \aaoTTM{} compared to \isTTM{}. \label{fig:ttmComp}}
		\end{subfigure}
		\renewcommand{\datafile}{data/dimTree_comp.dat}
		\tikzexternaldisable
		\tikzsetnextfilename{dimTree_comp}
		\begin{subfigure}{.5\textwidth}
			\centering

\resizebox{.97\textwidth}{!}{
\begin{tikzpicture}
  \pgfplotsset{compat=1.15}
  \begin{axis}[
    axis lines*=left,
    ymajorgrids,
    yminorgrids,
    minor tick num=4,
    ymin=0, 
    ymax= 3.5,
    ylabel = {Seconds},
    ylabel near ticks,
    ybar stacked,
    bar width=11pt,
    xtick=data,
    xmin=-1,
    xmax=11,
    xtick={0.75, 4.75, 8.75},
    xticklabel style={align=left},
    xticklabels ={
      d=3\\n=4000\\s=20,
      d=4\\n=500 \\s=5, 
      d=5\\n=145 \\s=3,
    },
    xticklabel style={font=\small,anchor= north},
    visualization depends on=rawy \as \rawy,
    point meta={(\thisrow{noDimTree-comp}+\thisrow{noDimTree-comm}+\thisrow{noDimTree-others})/(\thisrow{dimTree-comp}+\thisrow{dimTree-comm}+\thisrow{dimTree-others}) },
    nodes near coords style={
        black,
        font=\scriptsize,
        shift={(axis direction cs:0,\rawy/2)},
        anchor=south,
    },
    cycle list={
      {fill=red!80, draw=red!80},
      {fill=blue!80, draw=blue!80},
      {fill=gray!80, draw=gray!80}
    }
  ]
  \addplot table [x=dimTree-ind, y=dimTree-comp] {\datafile};
  \addplot table [x=dimTree-ind, y=dimTree-comm] {\datafile};
  \addplot+ [nodes near coords] table [x=dimTree-ind, y=dimTree-others] {\datafile};
  \end{axis}

  \begin{axis}[
    axis lines*=left,
    ymin=0, 
    ymax= 3.5,
    ybar stacked,
    bar width=11pt,
    xtick=data,
    xmin=-1,
    xmax=11,
    ticks=none,
    legend entries={noDimTree-comp, noDimTree-comm, noDimTree-others, dimTree-comp, dimTree-comm, dimTree-others},
   legend style={at={(1.3,.7)},anchor = east},
    legend cell align={left},
    cycle list={
      {fill=red!50, draw=red!50},
      {fill=blue!50, draw=blue!50},
      {fill=gray!50, draw=gray!50},
      {fill=red!80, draw=red!80},
      {fill=blue!80, draw=blue!80},
      {fill=gray!80, draw=gray!80}
    }
  ]
  \addplot table [x=noDimTree-ind, y=noDimTree-comp] {\datafile};
  \addplot table [x=noDimTree-ind, y=noDimTree-comm] {\datafile};
  \addplot table [x=noDimTree-ind, y=noDimTree-others] {\datafile};
  \addplot table [x=noDimTree-ind, y=fake]        {\datafile};
  \addplot table [x=noDimTree-ind, y=fake]        {\datafile};
  \addplot table [x=noDimTree-ind, y=fake]        {\datafile};
  \end{axis}
\end{tikzpicture}}

			\caption{Performance gain by using the dimension tree optimization for \cref{alg:rhosvdrekron}. We compute the sketch of a cubic synthetic $n$-way tensor where each mode has the same size $n$, resulting in a cubic $d$-way core tensor with each mode having the same size $s$. The labels on top of the bars show speedup gained by using dimension trees. \label{fig:dimTree_Comp}}
		\end{subfigure}	
		\caption{Performance benefits of our multi-TTM and dimension tree optimizations}
\end{figure}

\subsection{Strong scaling on synthetic data}\label{sec:strong_scaling}
In this experiment we benchmark four variations of our randomized algorithms and STHOSVD as a baseline, scaling from 2 nodes (64 cores) to 32 nodes (1024 cores) on a fixed problem size.
The input tensor is a $410\times 410 \times 410 \times 410$ single-precision synthetic tensor ($\approx$ 113 GB) constructed from multiplying a $20\times 20 \times 20\times 20$ randomly generated core with four $410 \times 20$ random matrices. No noise is added to this synthetic tensor so it is exactly low rank. This size is close to the largest tensor we can fit in the memory of 2 nodes, which makes the timing results more consistent and less influenced by noise in the system. Any order of modes used to compute the multi-TTM will not affect the performance as the tensor size and rank are the same across modes. All six algorithms are given target ranks $(20,20,20,20)$ and the randomized ones use oversampling parameter $p=3$. 

The results are presented in \cref{fig:strong_synthetic}, where we see that all six algorithms presented scale well to 1024 cores. 
All randomized algorithms outperform the deterministic \STHOSVD{}, and the randomized algorithms that use Kronecker-structured random matrices outperform the parallel version of \cref{alg:rsthosvd}.
A noticeable speedup of 3--4$\times$ is achieved by using \cref{alg:para_rsthosvdkron} compared to \STHOSVD{}. 
The second-best algorithm is \cref{alg:para_rhosvdkronreuse} with the dimension tree optimization achieving 2--3$\times$ speedup. 
Note that the use of \aaoTTM{} provided little improvement over \isTTM{} in this case, because the compression ratio is very high so the multi-TTMs are not bottlenecks. 
We can also see that the dimension tree optimization does provide noticeable and consistent speedup for \cref{alg:para_rhosvdkronreuse}.

\begin{figure}[!ht]
	\centering
	\renewcommand{\datafile}{data/strong_scaling_algs_min.dat}
	\tikzexternaldisable
	\tikzsetnextfilename{strong_scaling}


\resizebox{.5\textwidth}{!}{
\begin{tikzpicture}
  \begin{axis}[
    xlabel={Number of cores},
    ylabel={Seconds},
    ylabel near ticks,
    legend style={at={(1.45,.75)},anchor = east},
    ymajorgrids=true,
    xmajorgrids=true,
    grid style=dashed,
    xmode = log,
    log basis x = {2},
    log ticks with fixed point,
    xticklabels={0, 64, 128, 256 ,512 ,1024},
    ymode = log,
    log basis y = {2},
    legend cell align={left}
	]
    \addplot+[black] table [x=nprocs, y=sthosvd]{\datafile};
    \addplot+[red] table [x=nprocs, y=rsthosvd]{\datafile};
    \addplot+[yellow] table [x=nprocs, y=rsthosvd-kron-ttm]{\datafile};
    \addplot+[brown] table [x=nprocs, y=rsthosvd-kron-dcttm]{\datafile};
    \addplot+[teal] table [x=nprocs, y=rhosvd-kron-dcttm]{\datafile};
    \addplot+[magenta] table [x=nprocs, y=rhosvd-rekron-dcttm-dimTree]{\datafile};
    \legend{
      \STHOSVD, 
      parallel \ref{alg:rsthosvd}, 
      \ref{alg:para_rsthosvdkron} without AAO-mTTM, 
      \ref{alg:para_rsthosvdkron}, 
      \ref{alg:para_rhosvdkronreuse} without dimTree, 
      \ref{alg:para_rhosvdkronreuse}}
  \end{axis}
\end{tikzpicture}}
	\caption{Strong scaling of different algorithm variants}
	\label{fig:strong_synthetic}
\end{figure}
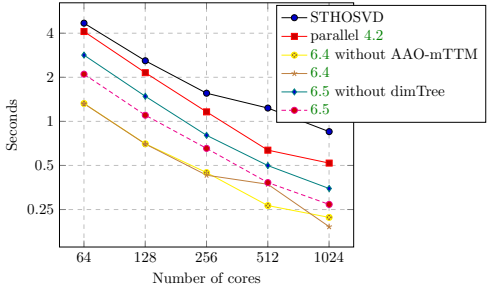

\subsection{Miranda dataset}\label{sec:miranda}
The Miranda dataset \cite{cabot2006reynolds,zhao2020sdrbench} contains three-dimensional simulation data of the density ratios of a non-reacting flow of viscous/diffusive fluids. This dataset is $3072 \times 3072 \times 3072$. Its values are in single-precision and range between 1 and 3. We show visualizations for this dataset in \cref{sec:sm_visuals}. 

In this experiment, we compare five variations of our randomized algorithms as well as the deterministic STHOSVD algorithm using the Miranda dataset. 
We use 4 nodes (128 cores) organized as a $1\times 8\times 16$ tensor. 
The target rank we choose for this run is $(502,\, 504,\, 361)$, which corresponds to a $10^{-2}$ reconstruction error (estimated using pre-computed singular values of the unfolding of the data tensor for each mode). 
The subrank matrix we used for \cref{alg:para_rsthosvdkron} is $\small \begin{bmatrix}1 & 39& 13 \\ 30& 1 & 17 \\ 6 & 61& 1 \end{bmatrix}$. 
For \cref{alg:para_rhosvdkronreuse}, the subrank vector we used is $(20,\, 20,\, 26)$. 
The relative error achieved by STHOSVD is 0.0094, while the relative errors of the randomized algorithms \cref{alg:rsthosvd,alg:para_rsthosvdkron,alg:para_rhosvdkronreuse} are 0.0234, 0.0194, and 0.0189, respectively, which are all within $2.5\times$ of the deterministic error.

The performance results are recorded in \cref{fig:miranda_comp}, and we visualize reconstructed tensors from \cref{alg:sthosvd,alg:rsthosvd,alg:para_rhosvdkronreuse} in \cref{sec:sm_visuals}.
First, we note that \STHOSVD{} is particularly slow when compared to the randomized algorithms, in this case partly because the Miranda dataset has fewer modes and each mode is large. 
As a result, the Gram matrices are large and the eigendecompositions of those Gram matrices are expensive. 
Moreover, the eigendecompositions are carried out redundantly on every processor due to the TuckerMPI assumption of small individual mode dimensions. 
(More details on the parallel implementation of \STHOSVD{} can be found in \cite{BKK20}). 
The randomized algorithms, on the other hand, can avoid this expensive step completely, and we see up to a $16\times$ speedup, comparing \cref{alg:para_rsthosvdkron} to \STHOSVD{}. 
In the next section, we compare the performance of all the algorithms again with a higher-order tensor where each mode is relatively small. 
In that case, the Gram matrices are smaller and the eigendecompositions are cheaper, so the deterministic algorithm appears more competitive. 

Also note that random number generation (forming the random matrices $\{\M{\Omega}\}$) in \cref{alg:rsthosvd} takes up a large percentage of the total time. 
These results demonstrate the benefits of generating fewer random numbers by using Kronecker-structured random matrices. 
Now, comparing the computation cost (red bar) of \cref{alg:rsthosvd} with that of the following algorithms to the right, we can see that using Kronecker-structured random matrices further reduces the computation cost of forming the factor matrices as we have predicted. 
Among the algorithms that use Kronecker-structured random matrices, \cref{alg:para_rsthosvdkron} and \cref{alg:para_rhosvdkronreuse} achieve the best performance. 
Comparing multi-TTM methods, we see that using either IS-mTTM or AAO-mTTM in \cref{alg:para_rsthosvdkron} results in very similar performance. 
Although AAO-mTTM achieved a 3$\times$ speedup in the communication cost of forming the factor matrices (pink bar) over IS-mTTM, the absolute speedup is not significant because the communication cost using IS-mTTM is already very small. 
This is mostly due to the subranks being very small compared to the size of the input tensor. 
Finally, we note that with our optimizations, applying the factor matrices to truncate the input tensor (the blue and light blue bars) now becomes the bottleneck of the algorithm.

\subsection{Stats-Planar dataset}\label{sec:sp}

The SP dataset is generated from the simulation of a statistically stationary planar (SP) methane-air flame \cite{Hemanth2016}. The data has dimensions $500\times 500\times 500\times 11 \times 400$, with the first 3 modes representing a 3D spatial grid, the 4th mode representing 11 variables, and the 5th mode representing time steps. This data has been used in previous studies such as \cite{BKK20} to demonstrate the effectiveness of Tucker decomposition algorithms. In this work, we use the single-precision-max-normalized version of this dataset. We visualize the 250th slice for each of the first three modes of the SP tensor in \cref{sec:sm_visuals}.

Similarly to the experiment on the Miranda dataset, we compare five variations of the randomized algorithm and the deterministic STHOSVD algorithm using 1024 cores (32 nodes). 
The target rank we use is $(31,\, 38,\, 35,\, 6,\, 11)$, which is the rank returned by the STHOSVD algorithm satisfying a $10^{-2}$ error tolerance. 
The subrank vector we used for \cref{alg:para_rhosvdkronreuse} is $\begin{bmatrix}
	2 & 2 & 2 & 3 & 4
\end{bmatrix}$ and the subrank matrix used for \cref{alg:para_rsthosvdkron} is $$\small \begin{bmatrix}
	0 & 3 & 3 & 2 & 2 \\
	3 & 0 & 4 & 2 & 2 \\
	2 & 5 & 0 & 2 & 2 \\
	2 & 2 & 2 & 0 & 2 \\
	2 & 2 & 2 & 2 & 0 
\end{bmatrix}.$$
The relative error achieved by STHOSVD on this dataset is 0.0028, while the relative errors of the randomized algorithms \cref{alg:rsthosvd,alg:para_rsthosvdkron,alg:para_rhosvdkronreuse} are 0.0050, 0.0079, and 0.0079, respectively, which are all within $3\times$ of the deterministic error.

The performance results are shown in \cref{fig:sp_comp}. 
The speedup of the randomized algorithms appears less dramatic compared to the results from the Miranda dataset, which is mainly due to the difference in dimensions of these two datasets. 
Recall that compared to the Miranda dataset, the SP tensor is of higher order but has a smaller size in each mode. 
As a result, the sequential eigendecomposition in the \STHOSVD{} algorithm is no longer as expensive. 
We also see that \cref{alg:para_rsthosvdkron,alg:para_rhosvdkronreuse} are still the best-performing algorithms. 
\Cref{alg:rsthosvd} suffers from slow random number generation, similar to the experiments on the Miranda dataset. 
Finally, forming the factor matrices is no longer a bottleneck for tensors with more cubical dimensions. 
In this case, accelerating the computations applying these factor matrices to derive the core tensor will become a more important issue.

\begin{figure}[!ht]
\renewcommand{\datafile}{data/Miranda_min_breakdowns.dat}
	\tikzexternaldisable
	\tikzsetnextfilename{miranda_comp}
	\begin{subfigure}{.45\textwidth}
		\centering

\resizebox{.9\textwidth}{!}{
\begin{tikzpicture}
  \begin{axis}[
		axis lines*=left,
    ymajorgrids,
    yminorgrids,
    minor tick num=4,
		ymin=0, ymax=15,
		ylabel = {Seconds},
		ylabel near ticks,
		ybar stacked,
		bar width=11pt,
		xtick=data,
		xticklabels ={
			\STHOSVD, \update{parallel \ref{alg:rsthosvd}}, \update{\ref{alg:para_rsthosvdkron} without AAO-mTTM}, \update{\ref{alg:para_rsthosvdkron}}, \update{\ref{alg:para_rhosvdkronreuse} without dimTree}, \update{ \ref{alg:para_rhosvdkronreuse}}
    },
		xticklabel style={font=\small,rotate=45,xshift=-2ex,yshift=-4ex,anchor=base east},
		legend entries={Form $\{\hat{\M{U}}\}$-comp,Form $\{\hat{\M{U}}\}$-comm,Form $\hat{\T{G}}$-comp,Form $\hat{\T{G}}$-comm,Form $\T{G}$, Rand matrix gen, Others},
    legend style={at={(1.1,.8)},anchor=east},
		legend cell align={left},
		cycle list={
			{fill=red!50, draw=red!50},
			{fill=red!30, draw=red!30},
      {fill=blue!50, draw=blue!50},
			{fill=blue!30, draw=blue!30},
      {fill=yellow!70, draw=yellow!70},
			{fill=gray!90, draw=gray!90},
      {fill=gray!50, draw=gray!50}
    },
    after end axis/.append code={
			\path [anchor=base east,yshift=0.5ex]
			(rel axis cs:0.125,0.95) node [yshift=1.9ex, xshift=-0.9ex] {$\vdots$};
			\path [anchor=base east,yshift=0.5ex]
			(rel axis cs:0.125,0.95) node [yshift=4.5ex] {96};}
		]
    \addplot table [x=algs, y=sketch-comp]{\datafile};
    \addplot table [x=algs, y=sketch-comm]{\datafile};
    \addplot table [x=algs, y=truncate-comp]{\datafile};
    \addplot table [x=algs, y=truncate-comm]{\datafile};
    \addplot table [x=algs, y=truncate-core]{\datafile};
    \addplot table [x=algs, y=rng]{\datafile};
    \addplot table [x=algs, y=others]{\datafile};
	\end{axis}
\end{tikzpicture}}
		\caption{Comparing the performance of all algorithms on Miranda dataset \label{fig:miranda_comp}}
	\end{subfigure}
	\tikzexternaldisable
	\tikzsetnextfilename{sp_comp}
	\renewcommand{\datafile}{data/SP_min_breakdowns.dat}
	\begin{subfigure}{.45\textwidth}
		\centering


\resizebox{.9\textwidth}{!}{
\begin{tikzpicture}
  \begin{axis}[
		axis lines*=left,
    ymajorgrids,
    yminorgrids,
    minor tick num=4,
		ymin=0, ymax=20,
		ylabel = {Seconds},
		ylabel near ticks,
		ybar stacked,
		bar width=11pt,
		xtick=data,
		xticklabels ={
			\STHOSVD, \update{parallel \ref{alg:rsthosvd}}, \update{  \ref{alg:para_rsthosvdkron} without AAO-mTTM}, \update{\ref{alg:para_rsthosvdkron}}, \update{\ref{alg:para_rhosvdkronreuse} without dimTree}, \update{ \ref{alg:para_rhosvdkronreuse}}
    },
		xticklabel style={font=\small,rotate=45,xshift=-2ex,yshift=-4ex,anchor=base east},
		legend entries={Form $\{\hat{\M{U}}\}$-comp,Form $\{\hat{\M{U}}\}$-comm,Form $\hat{\T{G}}$-comp,Form $\hat{\T{G}}$-comm,Form $\T{G}$,Rand matrix gen, Others},
    legend style={at={(1.1,.8)},anchor=east},
    legend cell align={left},
    cycle list={
			{fill=red!50, draw=red!50},
			{fill=red!30, draw=red!30},
      {fill=blue!50, draw=blue!50},
			{fill=blue!30, draw=blue!30},
      {fill=yellow!70, draw=yellow!70},
      {fill=gray!90, draw=gray!90},
      {fill=gray!50, draw=gray!50}
    }
		]
    \addplot table [x=algs, y=sketch-comp]{\datafile};
    \addplot table [x=algs, y=sketch-comm]{\datafile};
    \addplot table [x=algs, y=truncate-comp]{\datafile};
    \addplot table [x=algs, y=truncate-comm]{\datafile};
    \addplot table [x=algs, y=truncate-core]{\datafile};
    \addplot table [x=algs, y=rand-matrix-gen]{\datafile};
    \addplot table [x=algs, y=others]{\datafile};
	\end{axis}
\end{tikzpicture}}
		\caption{Comparing the performance of all algorithms on SP dataset \label{fig:sp_comp}}
	\end{subfigure}
	\caption{Performance breakdown of all algorithms on our two real datasets}
\end{figure}

\section{Conclusions and Future Work}\label{sec:conclusion}
We develop new randomized algorithms using a Kronecker product of random matrices that significantly decrease the computational cost in computing a Tucker decomposition. 
By accelerating the sketching step using Kronecker products, we remove the SVD as the dominant computational bottleneck. 
Our algorithms also reduce the number of random entries generated, which, as shown in our experiment results, could result in significant savings in the runtime compared to standard randomized algorithms. 
As the SVD step is no longer the dominant computation, future directions include accelerating the TTM computation, the other dominant portion of computing a Tucker decomposition, perhaps through a one-pass approach similar to \cite{sun2020low}. 
We develop probabilistic error bounds for our algorithms using SRHT matrices as they generalize to Kronecker products better than Gaussian matrices. 
The empirical results comparing Gaussian and SRHT matrices show that the error incurred from using SRHT matrices is not any worse than using Gaussian matrices. 
Our theoretical bounds are pessimistic in comparison, so there is room for improvement in the analysis, another potential future direction.
We implement our new randomized algorithms in parallel, developing a new algorithm that parallelizes the most expensive SVD component. 
Previous approaches such as \cite{choi2018high} parallelize other components, leaving the most expensive part to be computed locally.
Overall, we show in this work that choosing a random matrix that fits the structure of our problem is beneficial. 
The dense Gaussian matrix typically used in RandSVD in particular is not required, and performance is greatly improved by exploiting appropriate structure.

\appendix 
\section{Additional Algorithm: Randomized HOSVD with Kronecker product}
We now include the case of \cref{alg:rhosvdrekron} where we generate an independent set of Kronecker factors for each mode.

\begin{algorithm}[!ht]
\caption{Randomized HOSVD with Kronecker product}
\label{alg:rhosvdkron}
\begin{algorithmic}[1]\footnotesize
\Function{$[\T{G}, \{\M{U}_j\}] =$ \rHOSVDkron}{$\T{X},\V{r},p$}
\State Compute matrix of subranks $\M{S}$
\For{$j = 1:d$}
\State\label{line:omega} Draw $d-1$ random matrices $\M{\Phi}_{j,k} \in \mb{R}^{s_{j,k} \times n_k}$ for $k = 1,\dots, j-1,j+1,\dots, d$
\State\label{line:sketch_hkron} $\T{Y} \leftarrow \T{X} \times_1 \M{\Phi}_{j,1} \times \dots \times_{j-1} \M{\Phi}_{j,j-1} \times_{j+1} \M{\Phi}_{j,j+1} \times \dots \times_d \M{\Phi}_{j,d}$\label{line:rhosvdkron:sketchMultiTTM}
\State\label{line:qr_hkron} Compute thin QR $\Mz{Y}{j} = \M[\hat]{U}_j \M{R}$
\EndFor
\State\label{line:core_hkron}$\hat{\T{G}} = \T{X} \times_1 \M[\hat]{U}_1^\Tra \times \dots \times_d \M[\hat]{U}_d^\Tra$\label{line:rhosvdkron:formCoreMultiTTM}
\State\label{line:truncatecore_hkron} $[\T{G}, \{\M{V}_j\}] = $ STHOSVD$(\hat{\T{G}}, \V{r})$
\State\label{line:truncatefactor} $\M{U}_j = \hat{\M{U}}_j \M{V}_j$ for $j = 1,\dots, d$
\EndFunction
\end{algorithmic}
\end{algorithm}

\section{Detailed Complexity Analysis for \cref{alg:rsthosvdkron,alg:rhosvdrekron}}\label{sec:seq_complexity}
\paragraph{Cost of \cref{alg:rsthosvdkron}}
First consider \cref{alg:rsthosvdkron} (\rSTHOSVDkron). We form the sketch $\T{Y}$ for mode $j$ in \cref{line:sketch_stkron}, involving an intermediate core tensor as well as random matrices $\{\M{\Phi}_k\}_{k\neq j}^{d}$. The intermediate core at iteration $j$ has dimensions 
\vspace{-.1cm}
\begin{equation}\label{eq:coredim}
\underbrace{\ell \times \dots \times \ell}_{j-1} \times \underbrace{n \times \dots \times n}_{d-j+1}, 
\vspace{-.2cm}
\end{equation}
and the matrices $\{\M{\Phi}_k\}_{k=1}^{j-1}$ have dimensions $s \times \ell$, while the matrices $\{\M{\Phi}_k\}_{k=j+1}^{d}$ have dimensions $n \times \ell$. We can compute the multi-TTM with any order of modes, but the most efficient in this particular case is to start with a mode in which the dimension is $n$. Without loss of generality, we thus compute the product in reverse order of modes. Then, we have two separate cases corresponding to the different dimensions of the $\M{\Phi}$ matrices. For modes $d$ to $j+1$, the cost is $\sum_{i=1}^{d-j+1} 2s^i \ell^{j-1} n^{d+2-i-j}$. The cost for modes $j-1$ to $1$ is $\sum_{i=1}^{j-2} 2\ell^{j-i}s^{d-j+i+1}$. 

In \cref{line:core_stkron}, we perform a single TTM to compute the intermediate core.
In mode $j$, we compute a TTM with core $\T{\hat{G}}$ with the same dimensions as in \cref{eq:coredim}, and factor matrix $\M[\hat]{U}_j^\top$ with dimensions $n \times \ell$. The cost of this product is $2\ell^jn^{d-j+1}$. 
Putting all modes together, the dominant costs of \cref{alg:rsthosvdkron} are
$$2 \sum_{j=1}^d \left( \ell^j n^{d-j+1} + \sum_{i=1}^{d-j+1} s^i \ell^{j-1} n^{d+2-i-j} + \sum_{i=1}^{j-2} \ell^{j-1} s^{d-j+i+1}\right).$$

\paragraph{Cost of \cref{alg:rhosvdrekron}}
Now considering \cref{alg:rhosvdrekron}, the sketch is computed in \cref{line:sketch_hkron} of the full tensor $\T{X} \in \mb{R}^{n \times n \times \dots \times n}$ and random matrices $\M{\Phi}_j \in \mb{R}^{s \times n}$ for $j = 1,\dots, d$. For any mode, the cost is $\sum_{i=1}^{d-1} 2s^i n^{d-i+1}$. Then without the dimension tree optimization, the cost for all modes is $d\sum_{i=1}^{d-1} 2s^i n^{d-i+1}$. Incorporating dimension trees as discussed in \cref{sec:dimension_tree} reduces the constant $2d$ to $4$, so the cost of \cref{line:sketch_hkron} in \cref{alg:rhosvdrekron} for all modes is $4\sum_{i=1}^{d-1} s^i n^{d+1-i}$. 

After all the modes are processed, we compute the core via a multi-TTM in \cref{line:core_hkron} in \cref{alg:rhosvdrekron} of tensor $\T{X}\in \mb{R}^{n \times n \times \dots \times n}$ and factor matrices $\M[\hat]{U}_j$ for $j = 1,\dots, d$. The cost of this multi-TTM, computing in any order of modes, is $\sum_{j=1}^d 2\ell^j n^{d+1-j}$. 
Combining all the costs, including $d$ sketch computations and core computations, the complexity of \cref{alg:rhosvdrekron} is
$$4 \sum_{i=1}^{d-1} s^i n^{d+1-i} + 2\sum_{j=1}^d \ell^j n^{d+1-j}.$$
Note that we do not include the core truncation step in \cref{line:truncatecore_stkron} of \cref{alg:rsthosvdkron} or \cref{line:truncatecore_hkron} of \cref{alg:rhosvdrekron} in the dominant costs as the leading order term in this cost is only $\ell^{d+1}$ for all algorithms. We also ignore the cost of the thin QR factorizations in \cref{line:qr_hkron} in \cref{alg:rhosvdrekron} and \cref{line:qr_stkron} in \cref{alg:rsthosvdkron} as that cost is $2d\ell^2n$ in all three algorithms.

\section{Bounds on Core Singular Values}\label{sec:appendix_core}

\begin{lemma}\label{lem:core_svals}
Let tensors $\T{X} \in \mb{R}^{n_1 \times \dots \times n_d}, \T{Y} \in \mb{R}^{\ell_1 \times \dots \times \ell_d}$ and matrices $\M{U}_j \in \mb{R}^{n_j \times \ell_j}$ with orthonormal columns such that $\T{Y} = \T{X} \times_1 \M{U}_1^\top \times \dots \times_d \M{U}_d^\top$. Then 
\begin{equation*}
\sigma_i\left(\Mz{Y}{j}\right) \leq  \sigma_i\left(\Mz{X}{j} \right)
\end{equation*}
for each mode $j$ and for each $i = 1, \dots, \min\{\ell_j, \ell_j^\oslash\}$.
 \end{lemma}
\begin{proof}
First consider an $m\times n$ matrix $\M{B}$ and matrix $\M{A} = \M{B} \M{V}$ for $\M{V}$ with $k$ orthonormal columns. 
We can show that $\sigma_i(\M{A}) \leq \sigma_i(\M{B})$ for any index $i=1,\dots,\min\{m,k\}$. 
Consider first that $\M{A}$ is a submatrix of $\M{B}\bmat{\M{V} & \M{V}_\perp}$, implying that $\M{A}^\top \M{A}$ is a principal submatrix of . Then we can directly apply the Cauchy interlacing theorem \cite[Theorem 4.3.28]{horn2012matrix} to see that $\sigma_i(\M{A}) = \lambda_i(\M{A}^\top \M{A}) \leq \lambda_i(\M{B}^\top \M{B}) = \sigma_i(\M{B})$.

The same argument applies to the transpose of $\M{A}=\M{U}^\top \M{B}$ for $\M{U}$ with orthonormal columns.
Thus, for $\M{U}$ and $\M{V}$ with orthonormal columns and $\M{A}=\M{U}^\top \M{B} \M{V}$, $\sigma_i(\M{A}) \leq \sigma_i(\M{B})$ for $i$ ranging from 1 to the minimum of the numbers of columns of $\M{U}$ and $\M{V}$.

Now, let $\M{A} = \Mz{Y}{j}$, recalling that $\Mz{Y}{j} = \M{U}_j \Mz{X}{j} (  \M{U}_d \otimes \M{U}_{d-1} \otimes \dots \otimes \M{U}_{j+1} \otimes \M{U}_{j-1} \otimes \dots \otimes \M{U}_1 )$. 
Recall that the Kronecker product of matrices with orthonormal columns also has orthonormal columns.
Thus, $\sigma_i(\Mz{Y}{j}) \leq  \sigma_i(\Mz{X}{j})$ for $i=1,\dots,\min\{\ell_j,\ell_j^{\oslash}\}$ and $j=1,\dots,d$.
\end{proof}

\section{Proof of \cref{lem:omega1}}\label{sec:appendix_omega1}

\begin{proof}
Let $\M{W} = \M{V}^\top_1 \in \mb{R}^{r \times n}$ for notational simplicity, and recall that $\M{\Omega} = \M{D}(\M{H}_1 \otimes \dots \otimes \M{H}_q)$ is the Kronecker product of independent SRHT matrices where $\M{D} \in \mb{R}^{n \times n}$ and $\M{H}_j \in \mb{R}^{n_j \times s_j}$ for every $j$, given $n = \prod_{j=1}^q n_j$ and $\ell = \prod_{j=1}^q s_j$. Define $\M{G} = (\M{W \Omega})(\M{W \Omega})^\top \in \mb{R}^{r \times r}$. Note that $\M{W \Omega}$ is equivalent to $\M{\Omega}_1 \in \mb{R}^{r \times \ell}$.

Our approach will focus on the elementwise representation of $\M{G}$ and will be composed of 3 main steps: first, we will express the elements of $\M{G}$ in terms of two summands $\M{M}$ and $\M{N}$ that can be bounded more easily; second, we will obtain a deterministic bound for $\|\M{M}\|_2$ and then bound $\mb{E}[N_{ij}^2]$, which is the bulk of the proof; and third, we use our result from the previous step in conjunction with Markov's inequality to obtain a concentration inequality for $\|\M{N}\|_2$. 
Combining all these pieces will then give us the desired bound. 
Note that our three main steps follow the approach of \cite{matrixproof}, which analyzes the case where $\M{\Omega}$ is a single SRHT matrix. 

Define $\M{H} = \M{H}_1 \otimes \M{H}_2 \otimes \dots \otimes \M{H}_q$, letting the Kronecker product of subsampled Hadamard matrices be $\M{H}$ for ease of notation. 
Then we have another way to express $\M{G} $ as
\vspace{-.2cm}
\begin{equation}\label{eq:expandedG}
	\M{G} = (\M{W \Omega})(\M{W \Omega})^\top = \M{WDHH}^\top \M{DW}^\top = \M{WDFDW}^\top,
	\vspace{-.1cm}
\end{equation}
letting $\M{F} = \M{HH}^\top$. There are some important properties of $\M{D}$ and $\M{F}$ we will need, which we now explore. 
The diagonal matrix $\M{D}$ has i.i.d.\ entries drawn from the Rademacher distribution so that $\mb{E}D_a = 0$ and $D_a^2 = 1$ for $a = 1,\dots,n$.

Each element of $\M{F}$ can be written as a product of the entries of individual Gram matrices $\M{H}_j\M{H}_j^\top$. Specifically,
$F_{ab} = (\M{H}_1\M{H}_1^\top)_{i_1 j_1}(\M{H}_2\M{H}_2^\top)_{i_2 j_2}\cdots (\M{H}_q\M{H}_q^\top)_{i_q j_q}$, 
with $a$ the linear index with respect to $i_1,\dots,i_q$ and $b$ the linear index with respect to $j_1,\dots, j_q$. This representation allows us to break dependent expressions down into their independent parts, as each $\M{H}_i$ is independent from $\M{H}_j$ when $i \neq j$. We can then write the expectation of $F_{ab}$ as 
\begin{equation}\label{eq:expF}
	\mb{E} F_{ab} = \mb{E} (\M{H}_1\M{H}_1^\top)_{i_1 j_1} \mb{E}(\M{H}_2\M{H}_2^\top)_{i_2 j_2}\cdots \mb{E}(\M{H}_q\M{H}_q^\top)_{i_q j_q}.
\end{equation}
From \cite[Eqn. 38]{matrixproof}, we have that $\mb{E} (\M{H}_k\M{H}_k^\top)_{i_k j_k} = 0$ for $i_k \neq j_k$. If $a \neq b$, then $i_k \neq j_k$ for at least one $k$. Combining this and \cref{eq:expF}, we can say that $\mb{E} F_{ab} = 0$ for $a \neq b$. Now consider the case where $a = b$. From \cite[Eqn. 37]{matrixproof}, we have that $(\M{H}_k\M{H}_k^\top)_{i_k i_k} = s_k/n_k$ deterministically. Then, 
\vspace{-.2cm}
\begin{equation}\label{eq:Faa}
	F_{aa} = \prod_{k=1}^q \frac{s_k}{n_k} = \frac{\ell}{n}.
	\vspace{-.1cm}
\end{equation} 
The last piece we will need is $\mb{E} [F_{ab}^2]$. 
From \cite[Eqn. 39]{matrixproof}, $\mb{E} [(\M{H}_k\M{H}_k^\top)_{i_k j_k}^2] = s_k/n_k^2$ for $i_k \neq j_k$, and $\mb{E}[(\M{H}_k\M{H}_k^\top)_{i_k i_k}^2] = s_k^2/n_k^2$ from above.
If $a \neq b$, then $i_{k'} \neq j_{k'}$ for at least one $k'$. Combining this and \cref{eq:expF}, 
\begin{equation}\label{eq:expF2}
	\mb{E}[F_{ab}^2] \leq \frac{s_{k'}}{n_{k'}^2}\prod_{\substack{k =1 \\ k\neq k'}}^q \frac{s_{k}^2}{n_{k}^2}=\frac{\ell^2}{s_{k'}n^2} \leq \frac{\ell^2}{\min_k\{s_k\}n^2} = \frac{\ell^2}{s_k^* n^2},
\end{equation}
letting $s_k^* = \min_k\{s_k\}$.

With all these pieces in mind, we begin the main steps of the proof. 
We start with an elementwise representation of $\M{G}$, using the form in \cref{eq:expandedG}, with
	\vspace{-.2cm}
	\begin{align*}
		G_{ij} &= \sum_{a,b=1}^n W_{ia} D_a F_{ab} D_b W_{jb} \\ &= \sum_{a=1}^n W_{ia}W_{ja} D_a^2 F_{aa} + \sum_{a=1}^n W_{ia} D_a \sum_{\substack{b=1 \\ b\neq a}}^n W_{jb} D_b F_{ab},
		\vspace{-.2cm}
	\end{align*}
	for $1\leq i,j \leq r$.
	Consider the first term, where we have isolated the case $a=b$. 
	As the rows of $\M{W}$ are orthonormal, $D_a^2 = 1$, and $F_{aa} = \ell/n$, $G_{ij}$ can be written as $G_{ij} = \frac{\ell}{n}\delta_{ij} + \sum_{a=1}^n W_{ia} D_a \sum_{\substack{b=1 \\ b\neq a}}^n W_{jb} D_b F_{ab}$,
	where $\delta_{ij}$ is the Kronecker delta which is 1 when $i = j$ and 0 otherwise.
	Defining $\M{M} \in \mb{R}^{r \times r}$ to be the diagonal matrix with $\ell/n$ as each diagonal entry, and letting $\M{N} \in \mb{R}^{r \times r}$ be the matrix with entries 
	$N_{ij} = \sum_{a=1}^n W_{ia} D_a \sum_{\substack{b=1 \\ b\neq a}}^n W_{jb} D_b F_{ab}$,
	we have $\M{G} = \M{M}+\M{N}$. 
	
	 As $\M{M}$ is a diagonal matrix, we can easily see $\|\M{M}\|_2 = \ell/n$. Bounding $\|\M{N}\|$ is trickier; our approach will be to use the fact $\mb{E}\|\M{N}\|_2^2 \leq \mb{E}\|\M{N}\|_F^2 = \sum_{i,j =1}^r \mb{E}[N_{ij}^2]$ and first bound $\mb{E}[N_{ij}^2]$. We start by expanding the product
	\begin{equation}
	\label{eq:nijsq}
	\begin{aligned}
		&\mb{E}[N_{ij}^2] = \mb{E} \left( \sum_{a=1}^n W_{ia} D_a \sum_{\substack{b = 1 \\ b \neq a}}^n W_{jb} D_b F_{ab} \right) \left(\sum_{c=1}^n W_{ic} D_c  \sum_{\substack{f = 1 \\ f\neq c}}^n W_{jf} D_f F_{cf} \right) \\
		&= \mb{E} \sum_{a =1}^n W_{ia}^2 \left( \sum_{\substack{b = 1 \\ b \neq a}}^n W_{jb} D_b F_{ab}\right)^2 + \mb{E}  \sum_{\substack{ a,c = 1 \\ a \neq c}}^n W_{ia} W_{ic} D_a D_c \sum_{\substack{b=1 \\ b \neq a}}^n W_{jb} D_b F_{ab} \sum_{\substack{f = 1 \\ f \neq c}}^n W_{jf} D_f F_{cf},
	\end{aligned}
	\end{equation}
	which we have separated into the terms where $a = c$ and $a \neq c$. 
	Consider the first term of \cref{eq:nijsq}, where $a=c$. 
	As $\M{W}$ is a deterministic matrix, the expectation only affects the terms with $D_b$ and $F_{ab}$, so we have
	\vspace{-.3cm}
	\begin{equation}\label{eq:a=c}
		\mb{E} \sum_{a =1}^n W_{ia}^2 \left( \sum_{\substack{b=1 \\ b \neq a}}^n W_{jb} D_b F_{ab}\right)^2 = \sum_{a =1}^n W_{ia}^2 \mb{E} \left( \sum_{\substack{b=1 \\ b \neq a}}^n W_{jb} D_b F_{ab}\right)^2 .
	\end{equation}
	\vspace{-.4cm}
	
	\noindent Now consider the expectation portion of \cref{eq:a=c} for a fixed $1\leq a\leq n$. We can expand this product and distribute the expectation as
	\vspace{-.2cm}
		\begin{equation*}
		\begin{aligned}
			&\mb{E} \left( \sum_{\substack{b=1 \\ b \neq a}}^n W_{jb} D_b F_{ab}\right)^2 = \mb{E} \sum_{\substack{b=1 \\b\neq a}}^n W_{jb} D_b F_{ab} \sum_{\substack{f=1 \\f \neq a}}^n W_{jf} D_f F_{af} \\
			&= \sum_{\substack{b,f=1 \\b,f \neq a}}^n W_{jb}W_{jf} \mb{E} [D_b D_f F_{ab} F_{af}] = \sum_{\substack{b=1 \\b \neq a}}^n W_{jb}^2 \mb{E}[F_{ab}^2] + \sum_{\substack{b,f=1\\b,f\neq a \\ b\neq f}}^n W_{jb} W_{jf} \mb{E} [D_b D_f F_{ab} F_{af}]	
	\end{aligned}
	\end{equation*}
	where the last equality separates terms into where $b=f$ and where $b \neq f$, respectively. From \cref{eq:expF2} and as the rows of $\M{W}$ are normalized, we can write the term where $b=f$ as 
		$\sum_{\substack{b=1 \\ b \neq a}}^n W_{jb}^2 \mb{E}[F_{ab}^2] \leq \sum_{b=1}^n W_{jb}^2 \frac{\ell^2}{s_k^* n^2} = \frac{\ell^2}{s_k^* n^2}$.
	The term where $b \neq f$ can be written as
		$\sum_{\substack{b,f=1 \\ b,f\neq a \\ b\neq f}}^n W_{jb} W_{jf} \mb{E} [D_b D_f F_{ab} F_{af}] = \sum_{\substack{b,f=1 \\ b,f\neq a \\ b\neq f}}^n W_{jb} W_{jf} \mb{E}D_b \mb{E}D_f \mb{E} [F_{ab} F_{af}] = 0$
	from $\mb{E}D_b = 0$.
	These two results can be combined into the expectation portion of \cref{eq:a=c} to obtain
		$\sum_{a =1}^n W_{ia}^2 \mb{E} \left( \sum_{\substack{b=1 \\ b \neq a}}^n W_{jb} D_b F_{ab}\right)^2 \leq \sum_{a=1}^n W_{ia}^2 \frac{\ell^2}{s_k^* n^2} = \frac{\ell^2}{s_k^* n^2}$.
	
	We now focus on the second term of \cref{eq:nijsq}, where $a \neq c$. We split up both of the last two sums to extract the terms where $b=c$ and where $f = a$, giving
	\begin{equation*}
	\small
	\begin{aligned}
		&\mb{E}\sum_{\substack{a,c=1 \\ a \neq c}}^n W_{ia} W_{ic} D_a D_c \sum_{\substack{b=1 \\ b \neq a}}^n W_{jb} D_b F_{ab} \sum_{\substack{f=1 \\ f \neq c}}^n W_{jf} D_f F_{cf} = \\
		&\mb{E} \sum_{\substack{a,c = 1 \\ a \neq c}}^n W_{ia} W_{ic} D_a D_c \left( W_{jc} D_c F_{ac} {+} \sum_{\substack{b = 1 \\ a\neq b \neq c}}^n W_{jb} D_b F_{ab}\right) \left( W_{ja} D_a F_{ca} {+} \sum_{\substack{f=1 \\ a\neq f \neq c}}^n W_{jf} D_f F_{cf} \right).
	\end{aligned}
	\end{equation*}
	We expand this product into four terms we can bound separately, as
	\vspace{-.2cm}
	\begin{subequations}
	\small
	\begin{align}
		\mb{E} \sum_{\substack{a,c=1 \\ a \neq c}}^n W_{ia} W_{ic} D_a D_c &\left( W_{jc} D_c F_{ac} {+} \sum_{\substack{b=1 \\ a\neq b \neq c}}^n W_{jb} D_b F_{ab}\right) \left( W_{ja} D_a F_{ca} {+}  \sum_{\substack{f=1 \\ a\neq f \neq c}}^n W_{jf} D_f F_{cf} \right) \nonumber \\
		&=\mb{E} \sum_{\substack{a,c = 1 \\ a \neq c}}^n W_{ia} W_{ic} D_a D_c W_{jc} W_{ja} D_c D_a F_{ac} F_{ca} \label{eq:4parts1}\\
		&+\mb{E} \sum_{\substack{a,c = 1 \\ a \neq c}}^n W_{ia} W_{ic} D_a D_c \sum_{\substack{b = 1 \\ a\neq b \neq c}}^n W_{jb} D_b F_{ab} \sum_{\substack{f = 1 \\ a\neq f \neq c}}^n W_{jf} D_f F_{cf} \label{eq:4parts2}\\
		&+\mb{E} \sum_{\substack{a,c = 1 \\ a \neq c}}^n W_{ia} W_{ic} D_a D_c W_{jc} D_c F_{ac} \sum_{\substack{f = 1 \\ a\neq f \neq c}}^n W_{jf} D_f F_{cf} \label{eq:4parts3}\\
		&+\mb{E} \sum_{\substack{a,c = 1 \\ a \neq c}}^n W_{ia} W_{ic} D_a D_c W_{ja} D_a F_{ca} \sum_{\substack{b = 1 \\ a\neq b \neq c}}^n W_{jb} D_b F_{ab} \label{eq:4parts4}.
	\end{align}
	\end{subequations}
Consider \cref{eq:4parts1}. As $D_k^2 = 1$ and $\M{F}$ is symmetric, the expectation is just affected by $F_{ac}^2$. 
	We then have
	\begin{equation*}\label{eq:part1_exp}
	\begin{aligned}
		\mb{E} &\sum_{\substack{a,c = 1 \\ a \neq c}}^n W_{ia} W_{ic} D_a D_c W_{jc} W_{ja} D_c D_a F_{ac} F_{ca} = \sum_{\substack{a,c = 1 \\ a \neq c}}^n W_{ia} W_{ic} W_{jc} W_{ja} \mb{E} [F_{ac}^2] \\
		&\leq \frac{\ell^2}{s_k^* n^2} \sum_{\substack{a,c = 1 \\ a \neq c}}^n W_{ia} W_{ic} W_{jc} W_{ja} \leq \frac{\ell^2}{s_k^* n^2} \sum_{a,c = 1}^n W_{ia} W_{ic} W_{jc} W_{ja} \\ 
		&= \frac{\ell^2}{s_k^* n^2} \left( \sum_{a=1}^n W_{ia} W_{ja} \right)^2 = \frac{\ell^2}{s_k^* n^2}\left[ \M{WW}^\top \right]_{ij}^2 = \delta_{ij} \frac{\ell^2}{s_k^* n^2}.
	\end{aligned}
	\vspace{-.15cm}
	\end{equation*}
	In all three of the remaining parts, \cref{eq:4parts2}, \cref{eq:4parts3}, and \cref{eq:4parts4}, distributing the expectation to the independent random components gives us the expectation of the product of independent Rademacher entries. This means all three of these parts are equal to 0.

	With all these pieces, we now have 
		$\mb{E} [N_{ij}^2] \leq  \frac{\ell^2}{s_k^* n^2} + \delta_{ij}\frac{\ell^2}{s_k^* n^2}$.
	With the bound on $\mb{E} [N_{ij}^2]$, we can now bound the expectation of the norm of $\M{N}$:
		$\mb{E} \|\M{N} \|_2^2 \leq \mb{E} \|\M{N} \|_F^2 =  \mb{E} \sum_{i,j=1}^r N_{ij}^2 \leq \frac{\ell^2}{s_k^* n^2} \sum_{i,j=1}^r (1 + \delta_{ij}) = (r^2+r) \frac{\ell^2}{s_k^* n^2}$.
	Then, using Markov's inequality, $\| \M{N} \|_2^2 \leq \frac{\beta (r^2+r) \ell^2}{s_k^* n^2}$ with probability at least $1 - \frac{1}{\beta^2}$.

Now consider the term $\| \M{G}^\dagger \|_2$ . Recalling that $\M{G} = \M{M}+\M{N}$, we can express this instead as $\M{G} = (\M{I} + \M{NM}^{-1} ) \M{M}$.
Then we can write $\M{G}^\dagger = \M{M}^{-1}(\M{I} + \M{NM}^{-1})^\dagger$. Taking norms, we have
	$\| \M{G}^\dagger \|_2 \leq \| \M{M}^{-1} \|_2 \| (\M{I} + \M{NM}^{-1})^\dagger \|_2 \leq \frac{n}{\ell} \sum_{k=0}^\infty \| \M{NM}^{-1} \|_2^k$,
where we use the Taylor expansion $ (\M{I} + \M{NM}^{-1})^\dagger = \sum_{k=0}^\infty ( -\M{NM}^{-1} )^k$ (see \cite[Corollary 5.6.16]{horn2012matrix} for more details).
We can then write
	$\| \M{G}^\dagger \|_2 \leq \frac{n}{\ell} \sum_{k=0}^\infty \left( \| \M{N}\|_2 \| \M{M}^{-1}\|_2 \right)^k. $
We now consider $\|\M{N}\|_2 \| \M{M}^{-1}\|_2$ before the entire expression. As $\| \M{M}^{-1} \|_2 = n/\ell$,
	$\|\M{N}\|_2 \| \M{M}^{-1}\|_2 \leq \sqrt{\frac{\beta(r^2+r)}{s_k^*}}\leq 1 - \frac{1}{\alpha}$,
with probability at least $1-\frac{1}{\beta^2}$, where the last inequality comes from \cref{eq:alphabeta}. Then,
	$\| \M{G}^\dagger \|_2 \leq \frac{n}{\ell} \sum_{k=0}^\infty \left(1-\frac{1}{\alpha} \right)^k = \frac{n\alpha}{\ell}$.
Our smallest singular value is then
	$\frac{1}{\sigma_{\text{min}}^2(\M{W} \M{\Omega})} = \| \M{G}^\dagger \|_2 \leq \frac{\alpha n}{\ell}$,
with probability at least $1-\frac{1}{\beta^2}$.
Taking the square root, we obtain the desired result.
\end{proof}

\section{STHOSVD Error Analysis}\label{app:sthosvd_proof} 
\begin{theorem}
	Let $\T{T} = [\T{G}; \M{U}_1, \dots, \M{U}_d]$ be the approximation given by \cref{alg:rsthosvdkron} to $\T{X} \in \mb{R}^{n_1 \times \dots \times n_d}$ with target rank $\V{r} = (r_1,\dots,r_d)$ and oversampling parameter $p$. Let $\ell_j = r_j+p$ for $j = 1,\dots, d$. Then for sequences $\{\alpha_j\}_{j=1}^d$ and $\{\beta_j\}_{j=1}^d$ satisfying \cref{eq:alphabeta}, the following bound holds with probability at least $1- \sum_{j=1}^d \frac{1}{\beta_j^2}$,
	\begin{equation*}
		\| \T{X}- \T{T} \| \leq \left( \sum_{j=1}^d \left( 1 + \frac{\alpha_j n_j^\oslash}{\ell_j} \right) \sum_{i=r_j+1} \sigma_i^2 ( \Mz{X}{j} ) \right)^{1/2} + \left( \sum_{j=1}^d \sum_{i=r_j+1}^{\ell_j} \sigma_i^2 (\Mz{X}{j}) \right)^{1/2}.
	\end{equation*}
\end{theorem}
\begin{proof}
	The quantity $\varepsilon_{\text{core}}$ is bounded from \cref{eq:core_svals}, so we only consider $\varepsilon_{\text{rand}}$. Let $\T[\hat]{G}^{(j)} = \T{X} \times_1 \M[\hat]{U}_1^\top \times_2 \dots \times_j \M[\hat]{U}_j^\top$ be the partially truncated core tensor after processing mode $j$, and let $\hat{\T{T}}^{(j)} = \T[\hat]{G}^{(j)} \times_{1} \M[\hat]{U}_1 \times_2 \dots \times_j \M[\hat]{U}_j$ the resulting partial approximation to $\T{X}$. From the first equality of \cref{lem:proj}, we have
	\begin{equation*}
	\begin{aligned}
		\| \T{X} -\T[\hat]{T} \|^2 &= \| \T{X} - \T{X} \times_{1} \M[\hat]{U}_1 \M[\hat]{U}_1^\top \times_2 \dots \times_d \M[\hat]{U}_d \M[\hat]{U}_d^\top \|^2 \\
		&= \sum_{j=1}^d \| \T[\hat]{T}^{(j-1)} - \T[\hat]{T}^{(j)} \|^2 \\
		&= \sum_{j=1}^d \| \T[\hat]{G}^{(j-1)} \times_{1} \M[\hat]{U}_1 \times_2 \dots \times_{j-1} \M[\hat]{U}_{j-1} \times_j (\M{I}-\M[\hat]{U}_j \M[\hat]{U}_j^\top ) \|^2.
	\end{aligned}
	\end{equation*}
Consider the $j$-th term in the sum. We unfold to obtain
\begin{equation*}
\begin{aligned}
	 \| \hat{\T{T}}^{(j-1)} - \hat{\T{T}}^{(j)} \|^2 &= \| (\M{I} - \M[\hat]{U}_j \M[\hat]{U}_j^\top) \Mz[\hat]{G}{j}^{(j-1)} (\underbrace{\M{I} \otimes \dots \otimes \M{I}}_{d-j} \otimes \M[\hat]{U}_{j-1} \otimes \dots \otimes \M[\hat]{U}_1)^\top \|_F^2 \\ 
	&\leq \| (\M{I}-\M[\hat]{U}_j\M[\hat]{U}_j^\top) \Mz[\hat]{G}{j}^{(j-1)} \|^2,
\end{aligned}
\end{equation*}
as the columns of factor matrices $\M[\hat]{U}_j$ are orthonormal for all $j=1,\dots,d$.
We can then apply \cref{thm:matrixbound} to $\| (\M{I}-\M[\hat]{U}_j\M[\hat]{U}_j^\top) \Mz[\hat]{G}{j}^{(j-1)} \|$, giving
\begin{equation*}
	\| (\M{I}-\M[\hat]{U}_j\M[\hat]{U}_j^\top) \Mz[\hat]{G}{j}^{(j-1)} \|^2 \leq  \left(1+ \frac{\alpha_j n_j^\oslash}{\ell_j} \right)\sum_{i=r_j+1}^{n_j} \sigma_i^2(\Mz[\hat]{G}{j}^{(j-1)}),
\end{equation*}
except with probability at most $\frac{1}{\beta_j^2}$.
As $\T[\hat]{G}^{(j-1)}$ is a random quantity, we must bound this by the singular values of $\T{X}$. We can directly apply \cref{lem:core_svals} to $\T[\hat]{G}^{(j-1)} = \T{X} \times_1 \M[\hat]{U}_1^\top \times_2 \dots \times_j \M[\hat]{U}_j^\top$ relating the singular values of $\T[\hat]{G}^{(j-1)}$ to those of $\T{X}$. Thus,
 \begin{equation*}
	\| \T[\hat]{T}^{(j-1)} - \T[\hat]{T}^{(j)} \|^2 \leq \left(1+ \frac{\alpha_j n_j^\oslash}{r_j} \right)\sum_{i=\ell_j+1}^{n_j}\sigma_i^2(\Mz{X}{j}),
\end{equation*}
except with probability at most $\frac{1}{\beta_j^2}$. The failure probability for the entire sum is the union of all $d$ failure probabilities for each mode, bounded above by the sum of those probabilities by the union bound. Then,
\begin{equation*}
\begin{aligned}
	\| \T{X} - \T[\hat]{T} \|^2 &= \sum_{j=1}^d \| \T[\hat]{T}^{(j-1)} - \T[\hat]{T}^{(j)} \|^2 \\
	&\leq \sum_{j=1}^d \left(1+ \frac{\alpha_j n_j^\oslash}{\ell_j} \right)\sum_{i=r_j+1}^{n_j}\sigma_i^2(\Mz{X}{j}),
\end{aligned}
\end{equation*}
except with probability at most $\sum_{j=1}^d \frac{1}{\beta_j^2}$. Taking square roots gives $\varepsilon_{\text{rand}}$. Combining $\varepsilon_{\text{rand}}$ and $\varepsilon_{\text{core}}$, the total error is then
\begin{equation*}
\| \T{X}- \T{T} \| \leq \left( \sum_{j=1}^d \left( 1 + \frac{\alpha_j n_j^\oslash}{\ell_j} \right) \sum_{i=r_j+1}^{n_j} \sigma_i^2 ( \Mz{X}{j} ) \right)^{1/2} + \left( \sum_{j=1}^d \sum_{i=r_j+1}^{\ell_j} \sigma_i^2 (\Mz{X}{j}) \right)^{1/2}.
\end{equation*}
\end{proof}

\section{Accuracy} \label{sec:sm_accuracy}
We construct a different synthetic tensor for our second accuracy experiment. For this case, we generate $\T{X} \in \mb{R}^{500 \times 500 \times 500}$ to be a random 3-way tensor with true rank $(50,50,50)$ and added $10^{-4}$ relative Gaussian noise. We use $p=5$ for our oversampling parameter. 
In \cref{fig:rankplot}, we plot the relative error resulting from our algorithms with increasing target rank. 
We compare our algorithms to the deterministic STHOSVD as well as the randomized algorithm using one large random matrix. 
In the left plot, we show the relative error from all algorithms using Gaussian random matrices, and in the right plot, we show the relative error from our suggested algorithms using SRHT random matrices. 
In both plots, the relative error for all algorithms is large until we reach the true rank. 
At that point, the error drops close to the noise level, but the randomized algorithms have a higher relative error than the deterministic.
All the randomized algorithms are comparable to each other, with errors exceeding the deterministic algorithm by factors of $2$ to $7\times$.

\begin{figure}[!ht]
\centering
\tikzexternaldisable
\tikzsetnextfilename{relative_err_target_rank_gaus}
\resizebox{.45\textwidth}{!}{
\begin{tikzpicture}
    \begin{axis}[
    	title = {Gaussian Random Matrix},  
    	xlabel = {Target Rank},
    	ylabel = {Relative Error},
    	ymin = .00001, ymax = 1,
	ymode=log,
	xmin = 0, xmax = 100,
	legend entries = {\cref{alg:sthosvd}, \cref{alg:rhosvd}, \cref{alg:rsthosvd}, \cref{alg:rhosvdkron}, \cref{alg:rsthosvdkron}, \cref{alg:rhosvdrekron}},
	legend style={at={(0,-.25)},anchor=south west},
	xscale = .92, yscale = .8
    	]
    	\addplot[mark=none,color=blue,line width=1pt] table {figs/deter.dat};
	\addplot[mark=none,color=red,style=dashed,line width=1pt]  table {figs/simpG.dat};
	\addplot[mark=none,color=violet,style=dashed,line width=1pt] table {figs/simpS-G.dat};
	\addplot[mark=none,color=black,style=dashdotted,line width=1pt]  table {figs/kronG.dat};
	\addplot[mark=none,color=orange,style=dashdotted,line width=1pt] table {figs/kronS-G.dat};
	\addplot[mark=none,color=teal,style=dotted,line width=1.3pt]  table {figs/rekronG.dat};
    \end{axis}
\end{tikzpicture}
}
\tikzexternaldisable
\tikzsetnextfilename{relative_err_target_rank_srht}
\resizebox{.45\textwidth}{!}{
\begin{tikzpicture}
    \begin{axis}[
    	title = {SRHT Random Matrix},  
    	xlabel = {Target Rank},
    	ylabel = {Relative Error},
    	ymin = .00001, ymax = 1,
	ymode=log,
	xmin = 0, xmax = 100,
	legend entries = {\cref{alg:sthosvd}, \cref{alg:rsthosvdkron}, \cref{alg:rhosvdrekron}},
	legend style={at={(0,-.25)},anchor=south west},
	xscale = .92, yscale = .8
    	]
    	\addplot[mark=none,color=blue,line width=1pt] table {figs/deter.dat};
	\addplot[mark=none,color=orange,style=dashdotted,line width=1pt] table {figs/kronS-S.dat};
	\addplot[mark=none,color=teal,style=dotted,line width=1.3pt] table {figs/rekronS.dat};
    \end{axis}
\end{tikzpicture}}
\caption{Relative error of our randomized algorithms with Gaussian (left) and SRHT (right) random matrices as the target rank $(r,r,r)$ increases on the synthetic tensor with true rank $(50,50,50)$ and $10^{-4}$ relative Gaussian noise.}
\label{fig:rankplot}
\end{figure}
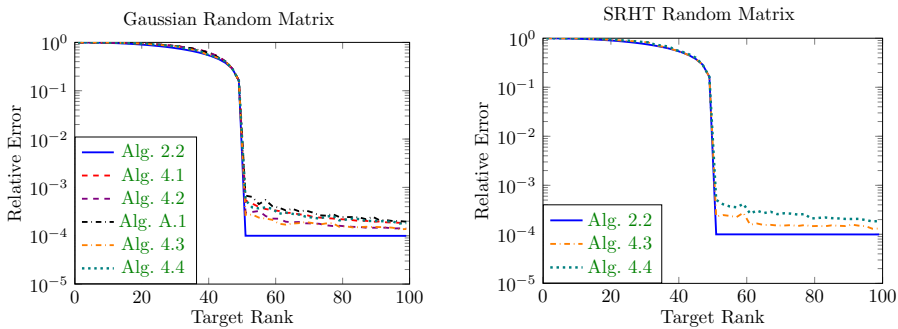

\section{Computation and Communication Cost Comparison with Previous Work}\label{sec:sm_cost}
In this section, we compare the leading terms of the per-processor computation and communication cost of performing the Tucker decomposition using the four algorithms listed in \cref{tab:parallel_cost_comp}. We will go over the expressions in each cell row by row. To reiterate notations, the input tensor is a $d$-way tensor $\T{X} \in \mb{R}^{n \times \dots \times n}$ with rank $(r,r,\dots, r)$ and the size of each mode of the $d$-way processor tensor is $q$. We also assume that $s < r < l \ll n$ where $s$ is the subrank for each mode and $\ell = r+p$ with $p$ the oversampling parameter. 

For the STHOSVD algorithm analyzed in \cite[Section 7]{BKK20}, the computation cost of forming the factor matrices includes computing the Gram matrices of the tensor unfoldings and the eigendecomposition of those Gram matrices. 
Computing the Gram matrices dominates the cost in this scenario. 
Furthermore, although there are $d$ Gram matrices, one for each mode, the latter $d-1$ Gram matrices are much cheaper to compute than the first; this reduction is due to the truncation occurring after processing each mode of the tensor. 
Therefore, when comparing the computation and bandwidth cost, we count only the first Gram matrix computation. 
For computation cost, the leading term is $\frac{n^{d+1}}{P}$, resulting from multiplying the local tensor unfolding of size $\frac{n}{q}\times (\frac{n}{q})^{d-1}$ with its own transpose. 
Leveraging the symmetry, the constant prefactor is 1. 
Considering bandwidth cost, computing the first Gram matrix is dominated by an all-to-all on the local input tensor which has $\frac{n^d}{P}$ elements. 
Since latency cost is not affected by truncation of the tensor, we count all $d$ Gram matrix computations. 
The point-to-point all-to-all collective is performed over the processor fiber, so the cost is $\mathcal{O}(q)=\mathcal{O}(P^{1/d})$ per mode. 

Forming the core tensor requires $d$ TTMs. 
Again, due to truncation, we count only the first TTM as the leading computational cost. 
In the first TTM, we multiply the local factor matrix of size $r\times \frac{n}{q}$ with the local tensor unfolding of size $\frac{n}{q} \times (\frac{n}{q})^{d-1}$, which leads us to the term $2\frac{rn^d}{P}$. 
The communication involves one reduce-scatter on the product of the local factor matrix and the local input tensor, leading us to the term $\beta\mathcal{O}(r (\frac{n}{q})^{d-1}) = \beta\mathcal{O}(\frac{rn^{d-1}}{P^{1-1/d}})$. 
For latency cost, there are $d$ reduce-scatters, one for each mode, over the processor fiber, yielding $\mathcal{O}(d\log q) = \mathcal{O}(\log P)$ messages.

For the algorithm proposed by Choi et al. in \cite{choi2018high}, the computational cost of forming the factor matrices is the same as stated above since they also compute the SVD through local eigendecomposition of the Gram matrices. 
The same applies to the computation cost of forming the core tensor.

The communication differs slightly. 
In this algorithm, for every other mode, an all-to-all among all the processors is performed to redistribute the global tensor unfolding into a 1D block-column fashion, which contributes $\alpha\mathcal{O}(dP)$ to the latency cost and $\beta\mathcal{O}(\frac{n^d}{P})$ to the bandwidth cost. 
Again we only count the first all-to-all for bandwidth due to truncation of the tensor in subsequent modes. 
Computing the Gram matrices requires an additional $d$ all-reduces among all the processors, which contributes $\alpha\mathcal{O}(d \log P)$ and $\beta\mathcal{O}(dn^2)$ to the latency and bandwidth cost respectively, but these are lower order terms. 
Forming the core tensor does not require additional communication given the redistributions described above. 

In \cref{alg:para_rsthosvdkron}, the computation cost of forming the factor matrices is dominated by the cost of forming the 1st factor matrix due to truncation, which is further dominated by the cost of multiplying the first random matrix of size $s\times \frac{n}{P^{1/d}}$ with the local tensor unfolding. 
Thus the leading term arrives at $2s\frac{n^d}{P}$, which has leading order term $2\frac{r^{1/(d-1)}n^d}{P}$. 
Using the all-at-once multi-TTM algorithm, the communication is that of a reduction of only the final result, which has dimensions $(n/P^{1/d})\times s \times \cdots \times s$, for a cost of $\alpha\mathcal{O}(\log P)+\beta\mathcal{O}(s^{d-1}n/P^{1/d})$.
This cost is consistent across modes.
Forming the core in \cref{alg:para_rsthosvdkron} has similar computation and communication cost to that of the STHOSVD algorithm except for the sizes of the factor matrices are $l\times \frac{n}{P^{1/d}}$, but the leading order term in the costs stays the same.
This assumes the in-sequence Multi-TTM algorithm is used, as in our implementation.

In \cref{alg:para_rhosvdkronreuse}, the computation cost of forming the factor matrices is double that of \cref{alg:rsthosvdkron} because the cost is dominated by the first two internal nodes of the dimension tree, each of which involves multiplying one random matrix with the local input tensor. 
The communication cost of forming the factor matrices is also similar to that of \cref{alg:para_rsthosvdkron}. 
For \cref{alg:para_rhosvdkronreuse}, the leading terms of the computation and communication cost of forming the core tensor are the same as that of \Cref{alg:para_rsthosvdkron}.

\section{Additional Experiments}\label{sec:sm_experiments}
\subsection{Miranda and SP visualizations}\label{sec:sm_visuals}
We show a three-dimensional visualization of a portion of the original Miranda tensor in \cref{fig:miranda3d}.

\begin{figure}[!ht]
\centering
\includegraphics[scale=.4]{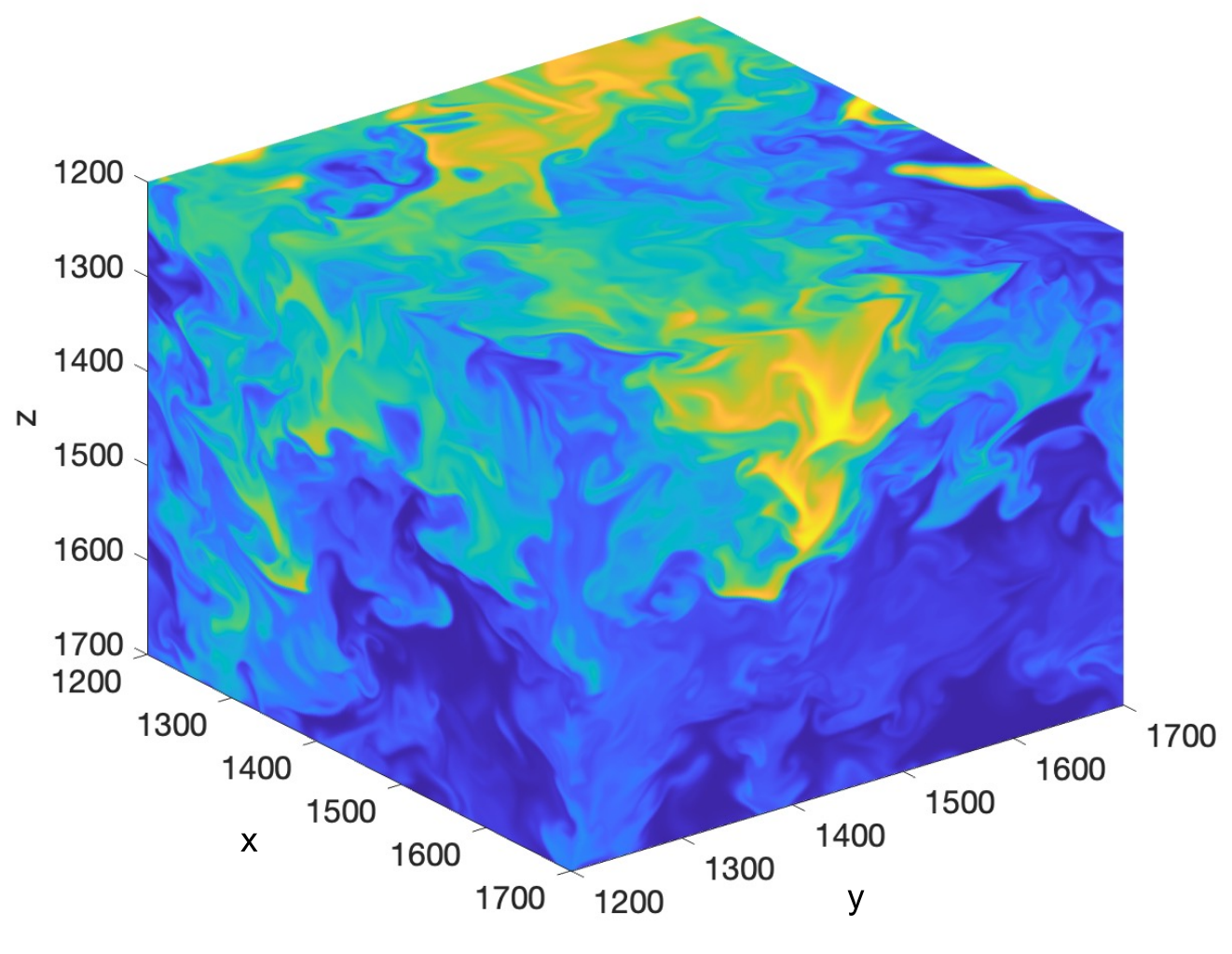}
\caption{Three-dimensional visualization of a subtensor of the Miranda tensor, with indices from 1200 to 1700 in each mode }
\label{fig:miranda3d}
\end{figure}

To validate the effectiveness of our algorithms on the Miranda tensor, we compare visualizations of a slice of the reconstructed tensor computed via \cref{alg:sthosvd,alg:rsthosvd,alg:para_rhosvdkronreuse}. The slices are shown in \cref{fig:miranda_reconstruct}. All three algorithms achieve a reconstruction error of 0.01 and very little difference can be seen between the results of the algorithms and the original data, shown in the first row. Note that we picked \cref{alg:para_rhosvdkronreuse} as a representative algorithm while \cref{alg:para_rsthosvdkron} and \cref{alg:rhosvdkron} with AAO-mTTM generate similar results.

\begin{figure}[!ht]
\centering
\tikzexternaldisable
\tikzsetnextfilename{miranda_reconstruct}
\includegraphics[scale=.8]{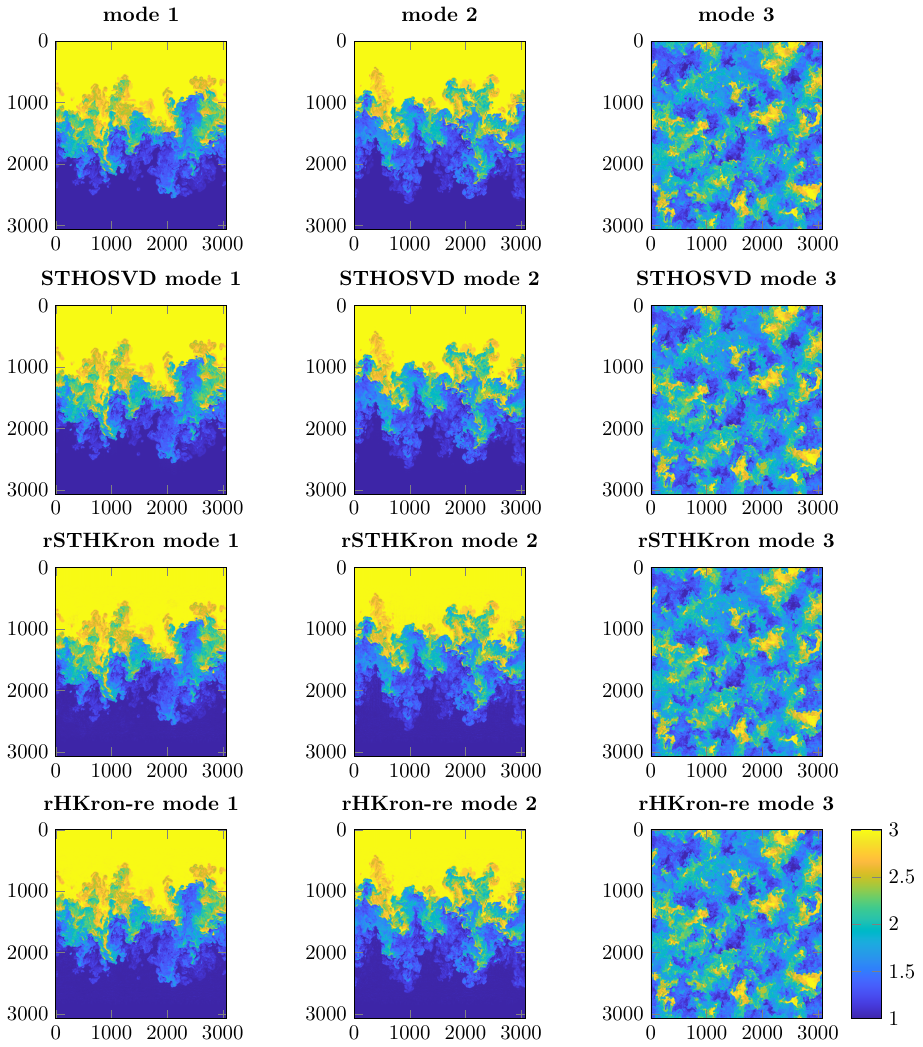}
\caption{1500-th slices of the original Miranda tensor for each mode (first row) compared to reconstructions of the 1500-th slices generated by \STHOSVD{} (second row), \cref{alg:para_rsthosvdkron} (third row), and \cref{alg:para_rhosvdkronreuse} (last row)}
\label{fig:miranda_reconstruct}
\end{figure}

We also visualize the slices of the the original SP tensor as well as of the tensors reconstructed from the decompositions computed by our randomized algorithms in \cref{fig:sp_reconstruct}. We can see that all of the algorithms preserve the significant features of the three slices although some compression artifacts can be seen. One reason why these artifacts are more noticeable, compared to the results we get from the Miranda dataset, is that we are zooming into a much smaller portion ($500\times 500$ instead of $3072\times 3072$) of the total tensor. 

\begin{figure}[!ht]
\centering
\tikzexternaldisable
\tikzsetnextfilename{sp_reconstruct}
\includegraphics[scale=.8]{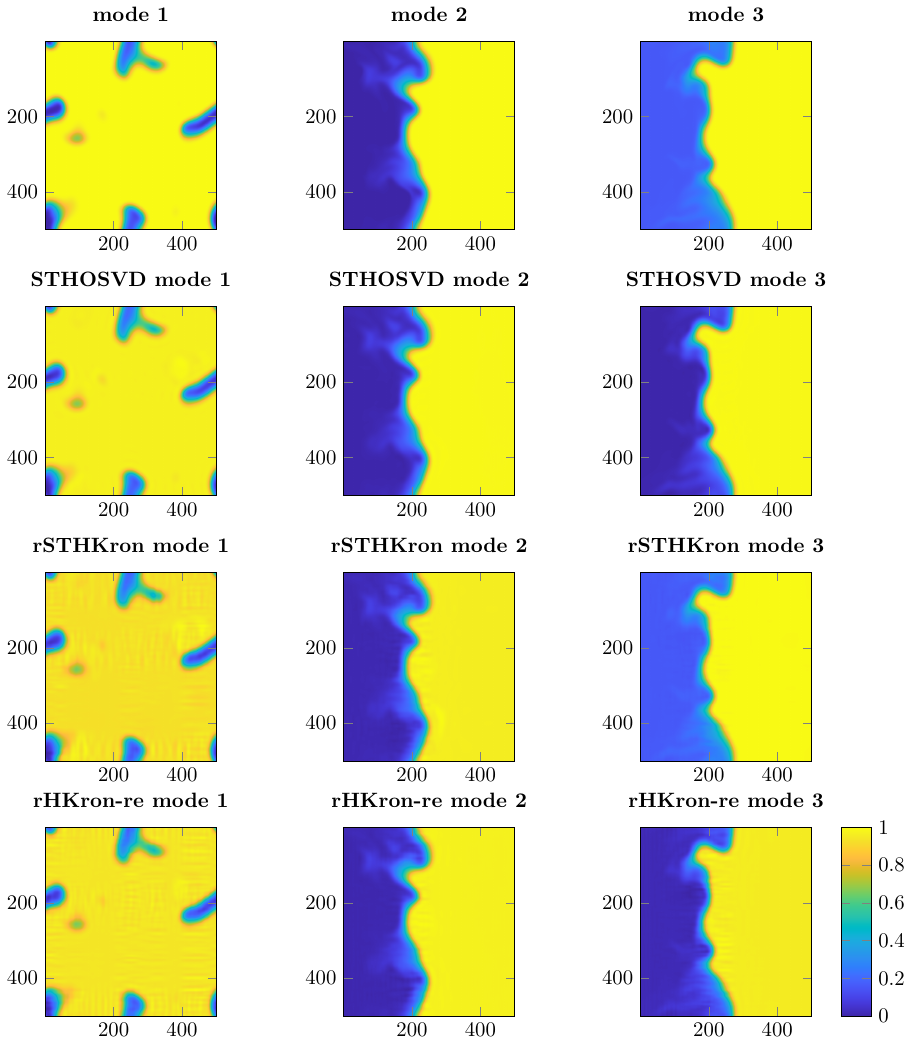}
\caption{250th slices of the original SP tensor for each mode (first row) compared to reconstruction of the 250th slices generated by \STHOSVD{} (second row), \cref{alg:para_rsthosvdkron} (third row), and \cref{alg:para_rhosvdkronreuse} (last row)}
\label{fig:sp_reconstruct}
\end{figure}

\bibliographystyle{siamplain}
\bibliography{refs}

\end{document}